\documentclass{conm-p-l}

\usepackage{latexsym}
\usepackage{amsmath}
\usepackage{amsfonts}
\usepackage{amssymb}
\usepackage{graphicx}
\usepackage{xcolor}

\usepackage[T1]{fontenc}
\usepackage[utf8]{inputenc}

\newtheorem{theorem}{Theorem}[section]

\newtheorem{proposition}[theorem]{Proposition}

\theoremstyle{definition}
\newtheorem{definition}[theorem]{Definition}
\newtheorem{example}[theorem]{Example}

\theoremstyle{remark}
\newtheorem{remark}[theorem]{Remark}

\numberwithin{equation}{section}

\newcommand{\rr}{\ensuremath{\mathbb{R}}}
\newcommand{\hh}{\ensuremath{\mathbb{H}}}
\newcommand{\sS}{\ensuremath{\mathbb{S}}}
\newcommand{\zz}{\ensuremath{\mathbb{Z}}}
\newcommand{\cc}{\ensuremath{\mathbb{C}}}
\newcommand{\nn}{\ensuremath{\mathbb{N}}}

\newcommand\re{\mathrm{Re}}
\newcommand\im{\mathrm{Im}}

\newcommand{\gen}{\mathrm{gen}\!\!}
\newcommand{\sing}{\mathrm{sing}}
\newcommand{\fp}{\mathop{\mathrm{fp}}}
\newcommand{\res}{\mathop{\mathrm{res}}}

\newcommand{\ii}{{\rm i}}
\newcommand{\sgn}{{\rm sgn}}
\newcommand{\e}{{\rm e}}
\newcommand{\dd}{{\rm d}}
\DeclareMathOperator{\sinc}{sinc}     
\DeclareMathOperator{\sinhc}{sinhc}     

\newcommand\LHS{\mathrm{LHS}}
\newcommand\RHS{\mathrm{RHS}}
\newcommand{\cC}{{\mathcal C}}
\newcommand{\cB}{{\mathcal B}}
\newcommand{\cG}{{\mathcal G}}
\newcommand{\cW}{{\mathcal W}}

\newcommand{\hospital}{de l'H\^opital}

\def\bbbone{{\mathchoice {\rm 1\mskip-4mu l} {\rm 1\mskip-4mu l}
{\rm 1\mskip-4.5mu l} {\rm 1\mskip-5mu l}}}
\def\one{\bbbone}

\begin{document}

\title[Generalized integrals of Macdonald and Gegenbauer functions]{Generalized integrals\\ of Macdonald and Gegenbauer functions}


%    author one information
\author{Jan Dereziński}
\address{Department of Mathematical Methods in Physics, Faculty of Physics, 
University of Warsaw, Pasteura 5, 02-093 Warszawa, Poland}
\curraddr{}
\email{jan.derezinski@fuw.edu.pl}
\thanks{The work of J.D. and C.G.
 was supported by the National Science Center of Poland under the
    grant UMO-2019/35/B/ST1/01651. J.D. would like to thank
    Howard Cohl for an inspiring discussion at an early stage of this
    work. We thank Micha{\l} Wrochna for pointing out 
    the references \cite{Gelfand64,Lesch97}.}

%    author two information
\author{Christian Gaß}
\address{Department of Mathematical Methods in Physics, Faculty of Physics, 
University of Warsaw, Pasteura 5, 02-093 Warszawa, Poland}
\curraddr{}
\email{christian.gass@fuw.edu.pl}
\thanks{}

%    author two information

\author{Błażej Ruba}
\address{Department of Mathematics, University of Copenhagen, Universitetsparken 5, DK-2100 Copenhagen Ø, Denmark}
\curraddr{}
\email{btr@math.ku.dk}
\thanks{}

%    The 2020 edition of the Mathematics Subject Classification is
%    the current definitive version.
\subjclass[2010]{33C05, 33C10, 47A52}

\date{\today}

\begin{abstract}
We compute bilinear integrals involving Macdonald and 
Gegenbauer functions. These integrals are convergent only for a limited 
range of parameters. However, when one uses {\em generalized
  integrals}, they can be computed essentially without restricting the
parameters.
The generalized integral is a~linear functional extending  
the standard integral to a certain class of functions
involving finitely many homogeneous non-integrable terms at the endpoints of the interval. 
For generic values of parameters, generalized bilinear integrals
of Macdonald and Gegenbauer functions can be obtained by  analytic
continuation from the region in which the integrals are convergent. In
the case of integer parameters we obtain expressions with explicit additional terms related to an \textit{anomaly}, namely the failure of the generalized integral to be scaling invariant.
\end{abstract}

\maketitle

%\marginpar{{\color{blue} I thanked Micha{\l} Wrochna for pointing 
%out \cite{Gelfand64,Lesch97}, see acknowledgement footnote.}}

\section{Introduction}
Consider a {\em Sturm-Liouville operator}
\begin{align} 
\cC:=-\rho(r)^{-1}\big(\partial_r p(r)\partial_r+q(r)\big)
\end{align} 
acting on functions on an interval $]a,b[$. $\cC$ is formally
symmetric for the bilinear scalar product with the {\em density} $\rho$:
\begin{align}\langle f|g\rangle:=\int_a^b f(r)g(r) \rho(r)\dd r.
\end{align}
Let $f$ be an eigenfunction of $\cC$, that is, $\cC f= E  f$.
In important applications one needs to know the value of the scalar product of $f$
with itself:
\begin{align}\label{square}
\langle f|f\rangle=\int_a^b f(r)^2\rho(r)\dd r.\end{align}
There exists a simple method, which often allows us to evaluate
\eqref{square}.
If $f_i$ are eigenfunctions corresponding to two different eigenvalues $E_i$, then
\begin{align}\label{square1}
\int_a^b f_1(r)f_2(r)\rho(r)\dd r&=\frac{\cW(f_1,f_2,b)-\cW(f_1,f_2,a)}{ E _1- E _2} , \\ \textup{where}\quad \cW(f_1,f_2,r):=& p(r)\big(f_1(r)f_2'(r)  -f_1'(r)f_2(r)\big).\end{align}
\eqref{square1} is sometimes called
{\em Green's identity}, or the {\em integrated  Lagrange identity}.
The rhs of \eqref{square1} can often be easily evaluated. Typically
this is possible if the endpoints $a,b$ are singular points of the
corresponding differential equation. We~will then say that $[a,b]$ is
a {\em natural interval} for the operator $\cC$.
\eqref{square1} is undefined for $f=f_1=f_2$, but one can often use a limiting procedure to derive
\eqref{square} from \eqref{square1}.
  We will call integrals of the form \eqref{square} or \eqref{square1} {\em bilinear integrals}.

In this paper we consider two families of Sturm-Liouville operators: the {\em Bessel operator} and the {\em Gegenbauer operator}:
\begin{align}
\cB_\alpha & :=-\partial_r^2-\frac{1}{r}\partial_r+\frac{\alpha^2}{r^2},
\label{bessel.} \\
\cG_{\alpha} & := - (1-w^2)\partial_w^2+2(1+\alpha )w\partial_w.
\label{gege0.}
\end{align}
$\cB_\alpha$ and $\cG_\alpha$ are very common in applications, e.g.\ in mathematical physics. Separation of variables in the Laplace equation on $\rr^d$ leads to the Bessel operator with $\alpha=\frac{d-2}{2}$, while separation of variables in the Laplace equation on the sphere $\sS^d$ and on the hyperbolic space $\hh^d$ lead to the Gegenbauer operator.

The  well-known {\em modified Bessel equation} is the eigenequation of
\eqref{bessel.} with the eigenvalue $-1$. The even better-known {\em (standard) Bessel
equation} is its eigenequation for the eigenvalue $1$.
$[0,\infty[$ is a  natural interval for $\cB_\alpha$, and the density
is $\rho(r)=2r$.  {\em Macdonald functions} are 
eigenfunctions of $\cB_\alpha$ decaying fast at infinity.

The {\em Gegenbauer equation} is the eigenequation of
\eqref{gege0.} with an eigenvalue that we parametrize as
$\lambda ^2-\big(\alpha +\frac{1}{2}\big)^2$. The natural intervals for
$\cG_\alpha$ are $[-1,1]$ with $\rho(w)=(1-w^2)^\alpha$, and
$[1,\infty[$ with $\rho(w)=(w^2-1)^\alpha$.

The main goal of our paper is to describe bilinear integrals involving  Macdonald and Gegenbauer functions.  These integrals are convergent only for a limited range of parameters. Therefore, we introduce a concept of the {\em generalized integral}, which allows us to compute \eqref{square} and \eqref{square1} for essentially all parameters.

The generalized integral is a linear functional on a certain class of
functions on the halfline. This class consists
 of functions integrable on $[1, \infty[$ whose restriction to $[0,1]$ is a linear combination of an integrable function and homogeneous functions. On integrable functions, the generalized and the usual (Lebesgue)
integral coincide.
The generalized integral of non-integrable homogeneous terms is defined as follows:
\begin{equation}
    \gen \int_0^1 r^\lambda \dd r = \begin{cases}
        \frac{1}{\lambda+1} & \text{for } \lambda \neq -1, \\
        0 & \text{for } \lambda = -1.
    \end{cases}
\end{equation}
Note that the case $\lambda=-1$ is troublesome.

We say that the
generalized integral is {\em anomalous} if at least one of coefficients at  negative
integer powers of $r$ is nonzero. In the anomalous case the generalized integral has more complicated properties. This case often appears in applications and is especially interesting.

The behavior of Macdonald and Gegenbauer functions near the endpoints 
of integration crucially depends on the the parameter $\alpha$. This 
parameter determines the index of the singular point of the Bessel/Gegenbauer 
equation. The bilinear integrals of Macdonald and Gegenbauer functions are 
convergent in the standard sense if and only if $|\re(\alpha)|<1$. For all other $\alpha$ 
they are well-defined only in the generalized sense. For
  $\alpha\in\zz\backslash\{0\}$, these integrals are anomalous.

%{\color{blue}\marginpar{I think this is enough physical motivation. I only added our reference after the first sentence 
%to indicate that the application is more than a vague speculation.}}
Let us describe the main application of bilinear integrals
\eqref{square}
of Macdonald and Gegenbauer functions that we have in mind 
\cite{DeGaRu}.
Consider a Schr\"odinger operator  on $\rr^d$, $\sS^d$ or $\hh^d$
with a potential confined to a very small region.
One often approximates its Green function (that is, the integral
kernel of its resolvent) by an expression derived by assuming that the
potential is supported at a point. This expression involves a bilinear
integral of Macdonald functions for $\rr^d$ and of Gegenbauer
functions for $\hh^d$ and $\mathbb{S}^d$. This is a usual convergent
integral in dimensions $d=1,2,3$, which corresponds to
$\alpha=-\frac12,0,\frac12$.
In these dimensions one obtains a 1-parameter family of self-adjoint realizations of the
Laplacian perturbed by a point-like potential. The situation is
different in dimensions $d\geq4$. The usual bilinear integral is
divergent and   point interactions do not lead to self-adjoint
realizations of the Laplacian.

One can show \cite{DeGaRu} that  in all dimensions
  Green functions of the Laplacian with
 perturbations supported in a small region are asymptotic to certain {\em
  renormalized Green functions}, which depend on a finite number of parameters. If $d\geq4$, except for the trivial (unperturbed) one,
these renormalized Green functions are not integral kernels of bounded
operators---hence the Laplacian perturbed by a point potential cannot
be interpreted as a self-adjoint operator. The parameters of
nontrivial renormalized
Green functions form naturally an affine space. The generalized integral leads
to a distinuished renormalized Green function---it provides
a reference point in the affine space of these parameters.

Note that  there is a considerable difference
between even and odd dimensions. In even dimension  bilinear
integrals are anomalous, while in odd dimensions they are not.
This corresponds to logarithmic terms in even dimensions, which are
absent in odd dimensions.

The case $d=4$, corresponding to $\alpha=1$, is of particular interest. 
It is the borderline
case: the bilinear integrals are well defined in the usual sense for
smaller $\alpha$, but for $\alpha=1$ one has to use their generalized version.
Besides, they are anomalous.

The above analysis is closely related to  renormalization
in Quantum Field Theory, especially based on the dimensional regularization.
After the Wick rotation, the Minkowski space becomes
the Euclidean space $\rr^4$, and the d'Alembertian becomes the
Laplacian. 
As described in  \cite{BG96} ,
the configuration space version of dimensional regularization 
in Quantum Field Theory can be carried out by allowing the 
order $\alpha$ of Bessel functions describing 
the propagators to be complex, so that certain integrals are 
well-defined in the usual sense.
Then one has to subtract the 
terms which diverge if $\alpha$ approaches the physical value 
and take the limit. Thus renormalization in Quantum Field Theory can
be interpreted as an application of  the (anomalous) generalized integral.
%In fact, dimensional regularization is very useful to compute anomalous generalized 
%integrals, as we shall see in Section \ref{ssc:dim_reg}.

%{\color{red}Bilinear integrals of Gegenbauer functions are also needed to treat contact interactions for Laplacians on spheres and hyperbolic spaces \cite{DeGaRu}, and for d'Alembertians on de Sitter and anti-de Sitter spactimes.}

\subsection{Remarks about literature}

The generalized integral  is closely related to the 
standard theory of extensions of {\em homogeneous distributions}, as described 
by H\"ormander \cite{Hoermander90}, and earlier by Hadamard 
\cite{Hadamard23,Hadamard32}, 
 Riesz \cite{Riesz}, and Gelfand 
\cite{Gelfand64}.

Various kinds of  generalized integrals appear in
the literature. 
One can divide them into two kinds: those involving
non-integrability at a finite point and those for which the problem
comes
from $\infty$. In our paper we use only the former setting: we assume
that the integrand is integrable close to infinity, however close to
the finite endpoint at the left hand side of the interval it may
involve non-integrable homogeneous terms.  In this setting definitions 
equivalent to our generalized integrable were given by Hadamard 
\cite{Hadamard23,Hadamard32} and Riesz \cite{Riesz}. A recent exposition 
of this topic can be found in a book by Paycha \cite{Paycha}, Chapter 1.

A very general discussion of the generalized integral is contained
in \cite{Lesch97}. The generalized integral is defined there
in terms of the Mellin transform, and one can account for 
sums of almost homogeneous singularities at 0 and $\infty$ 
simultaneously.

Let us mention that various forms of the generalized integrals where
the integrand is not integrable near $\infty$ are probably more common in the literature. Starting from Chapter 2, the book \cite{Paycha} is devoted 
mostly to this setting.
The noncommutative version of the generalized integral for polyhomogeneous pseudodifferential operators in the non-anomalous case is sometimes called the {\em Vishik-Kontsevich trace}. In the anomalous case it is related to the {\em Wodzicki residue}.

The Gegenbauer equation is essentially a special case of the hypergeometric
equation, where we put the finite singular points at $-1,1$ and we
assume the symmetry with respect to $w\to -w$.
Its particular solutions for special values of parameters are the
well-known {\em Gegenbauer polynomials}. The Gegenbauer equation is
equivalent through a simple transformation to the {\em associated Legendre
equation}. In the literature, {\em associated Legendre functions} are much
more common than Gegenbauer functions. However, the use of Gegenbauer 
functions with the conventions that we introduce, rather than associated 
Legendre functions, leads to a simplification of various  identities, and 
therefore seems to be preferable.

The application of bilinear integrals
to 
 Green's functions of the Laplacian  
 with point interactions
 will be described in a separate paper
\cite{DeGaRu}.  It is a classic result that Green's functions
without point interactions
can be expressed
 in terms of Macdonald functions for $\rr^d$ and Gegenbauer (or associated
 Legendre) 
 functions for $\hh^d$ and $\mathbb{S}^d$ (cf. e.g. \cite{CDT,Szmytkowski07}).
 Green's functions with point potentials
 in dimensions 1,2,3, at least in the flat case, are also well-known
 \cite{AGHH,AK,BF,LSSY,RSII}.
Green's functions with point potentials in all dimensions, and also
for $\sS^d$ and $\hh^d$, will be described in our subsequent paper
 \cite{DeGaRu}.

\subsection{Outline of the paper}
The main goal 
of the present article is to  pedagogically describe
elements of the theory the generalized integral and of Bessel and Gegenbauer equation needed in 
 \cite{DeGaRu}.

In Section \ref{sec:gen_ints}, we define the generalized integral and
study its properties. We describe its behavior
under a change of variables. We describe a method
of computing
generalized integrals which resembles the {\em dimensional
  regularization} in Quantum Field Theory, and actually can be traced
back to an earlier work of M.Riesz.

In Section \ref{sec:bessel} we recall elements of the theory of the
Bessel equation.
The main new result is the computation of generalized  bilinear
integrals of Macdonald functions using the definitions and methods of Section \ref{sec:gen_ints}.

Section \ref{sec:gegenbauer} is devoted
to the Gegenbauer  equation. We introduce two kinds of Gegenbauer
functions: ${\bf S}_{\alpha,\lambda}$ with a simple behavior near $1$
and ${\bf Z}_{\alpha,\lambda}$ with a simple behavior near $+\infty$. We compute
bilinear integrals of
${\bf S}_{\alpha,\lambda}$
and ${\bf Z}_{\alpha,\lambda}$,
both usual and generalized.
We discuss various kinds of integral representations of Gegenbauer
functions. These representations are then used to show that
Gegenbauer functions are asymptotic to 
Macdonald functions. We also prove that, if we choose the variables correctly, then 
 bilinear integrals of Gegenbauer functions are asymptotic to bilinear
 integrals of
Macdonald functions.

For obvious reasons, Section \ref{sec:gegenbauer} about the Gegenbauer
equation is more complex than Section \ref{sec:bessel}. However, both
sections have to a  large extent
a parallel structure.

This paper has two appendices. 
In Appendix \ref{app:Gamma}, we list several useful properties of  special 
functions related to the Gamma function, which are needed in the main
part of the paper. Appendix  
\ref{app:assoc_legendre} contains an overview of conventions for
functions related 
to Gegenbauer functions.

\section{Generalized integral}
\label{sec:gen_ints}

\subsection{Definition  of the generalized integral}
 In this section we introduce the {\em generalized
  integral}, which extends the usual integral to a certain class of not integrable functions. We restrict ourselves to functions which fail to be integrable near the (finite) left endpoint of the integration interval.
 Our definition  essentially coincides with similar concepts introduced
in \cite{Paycha}, Chapter 1, and goes back to the works of Riesz and Hadamard.

One should mention that it is also natural to define the generalized
integral for functions not integrable close to
$\infty$.  This  is described e.g. in \cite{Paycha} starting from
Chapter 2, and is actually more common in the literature. In our
work this will not be considered.

\begin{definition}
\label{def:genInt}
Let $a\in\rr$. We say that a function 
$f$ on $]a,\infty[$ is {\em integrable in the generalized sense} if it is 
integrable on $]a+1,\infty[$ and there exists a finite set $\Omega\subset\cc$ 
and complex coefficients $(f_k)_{k \in \Omega}$ such that 
\begin{align} 
f-\sum_{k\in\Omega} f_k(r-a)^k
\end{align} 
is integrable on $]a,a+1[$.  We define
\begin{align}
\label{gener}
&\gen\int_a^\infty f(r)\dd r
\\ \notag := &\sum_{k\in\Omega\backslash\{-1\}}\frac{f_k}{k+1}+
\int_a^{a+1}\Big(f(r)-\sum_{k\in\Omega}f_k (r-a)^k\Big)\dd r+\int_{a+1}^\infty
f(r)\dd r.\end{align}
\end{definition}

We note that the set 
$\{k \in \Omega \, | \, \re(k)\leq -1\}$ and the corresponding $f_k$ are uniquely determined by $f$. It is convenient to allow $k \in \Omega$ with $\re(k)>-1$. The generalized integral of $f$ does not depend on the choice of $\Omega$.

The generalized integral extends the standard integral:
\begin{align}
\gen\int_a^\infty f(r)\dd r=\int_a^\infty f(r)\dd  r \quad \text{for } f \in L^1[a, \infty[. \end{align}
 It is translation 
invariant in the following sense:
\begin{align}
\label{eq:genInt_transl}
\gen\int_a^\infty f(r) \dd r
&= \gen\int_{a-\alpha}^{\infty} f(u+\alpha) \dd u, \quad\alpha\in\rr.
\end{align}
In particular, we can without loss of generality assume $a=0$. 

If $f$ is defined on an interval $]0,b[$ with $b>0$, we define $\gen \int_0^b f(r) \dd r$ 
as the generalized integral of the extension of $f$ to $[0, \infty[$ by zero.

\begin{proposition} If $f$ satisfies Def. \ref{def:genInt} on $[0,\infty[$, then
\begin{align}\gen\int_0^\infty f(r)\dd
r=\lim_{\delta\searrow0}\Bigg(\int_\delta^\infty f(r)\dd
r+\sum_{k\in\Omega\backslash\{-1\}}\frac{f_k}{k+1}\delta^{k+1}+f_{-1}\ln(\delta)\Bigg).\end{align}
\end{proposition}

\begin{proof}
For any $0<\delta\leq1$ we have
\begin{align}\notag\gen\int_0^\infty f(r)\dd
r=&\int_\delta^\infty f(r)\dd
  r+\sum_{k\in\Omega\backslash\{-1\}}\frac{f_k}{k+1}\delta^{k+1}+f_{-1}\ln(\delta)\\
  &+\int_0^\delta\Big(f(r)-\sum_{k\in\Omega}f_kr^k\Big)\dd r.
\label{llast}    \end{align}
The last term of
\eqref{llast} converges to $0$ as $\delta \to 0$.
\end{proof}

\subsection{Change of variables}

The standard formula for a change of variables in an integral is not true for the generalized integral, but one may write down the correction terms explicitly.

\begin{proposition} 
Suppose that $g : [0,\infty[ \to [0,\infty[$ is a
bijection, smooth down to zero, such that $g(0)=0$ and $g'(0) \neq 0$. If $f$ is integrable in the generalized sense, then the same is true for $(f \circ g) g'$ and we have
%\marginpar{{\color{blue} $\dd r\to \dd u$ on LHS of \eqref{eq:change_of_var}}}
\begin{align}
    & \gen \int_0^\infty  f(g(u))g'(u) \dd u - \gen \int_0^\infty  f(r) \dd r  \label{eq:change_of_var} \\
    = &- f_{-1} \ln g'(0) + \sum_{l=2}^{\infty} \frac{f_{-l}}{(l-1)(l-1)!} \left. \frac{\dd^{l-1}}{\dd u^{l-1}}  \left( \frac{u}{g(u)} \right)^{l-1} \right|_{u=0}, \nonumber
\end{align}
where we put $f_k =0$ for $k \not \in \Omega$. In particular for $g(u)= \alpha u$, $\alpha >0$:
\begin{equation}
    \gen \int_0^\infty f(\alpha u) \alpha \dd u = \gen \int_0^\infty f(r) \dd r - f_{-1} \ln \alpha. \label{eq:genInt_scaling}
\end{equation}
\end{proposition}

\begin{proof}
The first statement easily follows from the Taylor expansion of $g$ around zero. Let $\Delta_f(g)$ be the left hand side of \eqref{eq:change_of_var}. We have
\begin{align}
\Delta_f(g) = \int_{g(1)}^1 f(r) \dd r + \gen \int_0^1 \left( f(g(r))g'(r) - f(r) \right) \dd r,
\label{eq:Delta_fg_1}
\end{align}
where we used the standard change of variables formula for the integral over $[1, \infty[$ and $\int_{g(1)}^1 - \int_1^\infty = \int_{g(1)}^1$. Now decompose $f(r) = F(r) + \sum_{k \in \Omega} f_k r^k$. It is easy to check that $F$ gives no contribution in \eqref{eq:Delta_fg_1}, so
\begin{equation}
\Delta_f(g) = \sum_{k \in \Omega} f_k \left( \int_{g(1)}^1 r^k \dd r + \gen \int_0^1 \left( g(r)^k g'(r) - r^k \right) \dd r \right).
\end{equation}
We will analyze this expression term by term. First consider $k \neq -1$. Then
\begin{align}
    &\int_{g(1)}^1 r^k \dd r + \gen \int_0^1 \left( g(r)^k g'(r) - r^k \right) \dd r \label{eq:change_var_termk} \\
    =&  - \frac{g(1)^{k+1}}{k+1} + \frac{1}{k+1} \gen \int_0^1 \frac{\dd}{\dd r} g(r)^{k+1} \dd r. \nonumber
\end{align}
We have a Taylor expansion of the form
\begin{equation}
    g(r)^{k+1} = r^{k+1} \left( \sum_{j=0}^N a_j r^j + R(r) \right),
\end{equation}
where $R(r) =O(r^{N+1})$ and $N$ is such that $\re(N+k+2) >0$. Then
\begin{align}
    & \gen \int_0^1 \frac{\dd}{\dd r} g(r)^{k+1} \dd r = \sum_{j=0}^N (k+j+1) a_j \, \gen \int_0^1 r^{k+j} \dd r + R(1) \\
    = & \sum_{\substack{j=0 \\ k+j \neq -1}}^N a_j + R(1) = g(1)^{k+1} - a_{-k-1}, \nonumber
\end{align}
where we put $a_{-k-1} = 0$ if $-k-1 \not \in \nn$. Therefore, \eqref{eq:change_var_termk} vanishes for $-k-1 \not \in \nn$. If $k=-l$ with $l \in \{ 2, 3, \dots \}$, then \eqref{eq:change_var_termk} is equal to
\begin{equation}
 - \frac{1}{-l+1} a_{l-1} =  \frac{1}{(l-1)(l-1)!} \left. \frac{\dd^{l-1}}{\dd r^{l-1}}  \left( \frac{r}{g(r)} \right)^{l-1} \right|_{r=0}.
\end{equation}
To handle the case $k=-1$, we write $g(r)= g'(0)r \e^{R(r)}$ with $R(0)=0$ and compute
\begin{equation}
    \gen \int_0^1 \frac{g'(r)}{g(r)} \dd r = R(1) = \ln g(1) - \ln g'(0).
\end{equation}
\end{proof}

 Formula \eqref{eq:genInt_scaling} shows that the generalized integral
is invariant under scaling if and only if $f_{-1} = 0$, and invariant
under a large class of a change of variables if $f_k=0$ for every negative integer $k$.

\begin{definition}
The generalized integral \eqref{gener} is called {\em anomalous} if
there exists  $n=1,2,\dots$ such that
$f_{-n}\neq0$.
\end{definition}

Remarkably, the generalized integral is always invariant under changes of variables given by power functions.

\begin{proposition}
\label{prop:genInts_scaling}
Let $f$ be integrable in the generalized sense. Then for $\alpha > 0$
\begin{equation}
    \label{eq:genInt_powers} \gen\int_0^\infty f(r)\dd r =\gen\int_{0}^\infty f( u^\alpha)\, \alpha u^{\alpha-1} \dd u .
\end{equation}
\end{proposition}
\begin{proof}
The formula is true for integrable $f$, so by linearity it is sufficient to verify it for $f(r)= r^k \one_{[0,1]}(r)$. This is an elementary calculation.
\end{proof}

\subsection{Dimensional regularization}
\label{ssc:dim_reg}

In this subsection we describe a method to compute generalized integrals by analytic continuation. It is closely related to dimensional regularization used in QFT.

Let $F$ be a holomorphic function on $U \setminus \{ z_0 \}$, where $z_0 \in U$ and $U$ is open. Then one has a Laurent expansion
\begin{equation}
    F(z) = \sum_{j \in \zz} F_j (z-z_0)^j
\end{equation}
convergent for $z \neq z_0$ sufficiently close to $z_0$. Coefficients $\res F(z_0) := F_{-1}$ and $\fp \limits_{z \to z_0} F(z) :=F_0$ are called the residue and the finite part of $F$ at $z_0$.

Let  $N\in\nn$ and let $f : \, ] 0 , \infty [ \times \{ \alpha \in \cc \, | \, \re(\alpha) > -N-1 \} \to \cc$ be a function such that $f(r , \cdot)$ is holomorphic for each $r$, $\| f(\cdot, \alpha) \|_{L^1[1, \infty[}$ is bounded locally uniformly in $\alpha$, and there exist holomorphic functions $f_0, \dots, f_N$ of $\alpha$ such that the $L^1]0,1]$ norm of $f(r,\alpha)-\sum_{n=0}^N r^{\alpha+n}f_n(\alpha)$ is bounded locally uniformly in $\alpha$. Then $f(\cdot, \alpha)$ is integrable in the generalized sense, and for $- \alpha \not \in \{ 1 , \dots, N \}$ one has
\begin{align} 
\label{anla}
&\gen\int_0^\infty f(r,\alpha)\dd r \\ \notag
= &  \sum_{n=0}^N \frac{f_{n}(\alpha)}{\alpha+n+1}
+ \int_0^1 \Big( f(r,\alpha) - \sum_{n=0}^N r^{\alpha+n} f_n(\alpha) \Big) \dd r
+ \int_1^\infty f(r,\alpha) \dd r.
\end{align}
By Morera's theorem, the right hand side is, away from the poles at $-1, \dots, -N$, a~holomorphic function of $\alpha$. Therefore, to obtain \eqref{anla} in the non-anomalous case it is enough to compute \eqref{anla} in
the region where the usual integral is convergent and continue analytically. 

Let $m \in \{ 1 , \dots, N \}$. The right hand side of \eqref{anla} has a simple pole at $\alpha = -m$ with residue $f_{m-1}(-m)$ (possibly zero). Its finite part is 
\begin{equation}
\operatornamewithlimits{fp}_{\alpha \to -m} \gen \int_0^\infty f(r,\alpha)\dd r
  =\lim_{\alpha\to-m}\Bigg(\gen\int_0^\infty f(r,\alpha)\dd
r-\frac{f_{m-1}(-m)}{\alpha+m}\Bigg).  
\end{equation}
One may be tempted to think that the finite part is equal to the generalized integral of $f(\cdot, -m)$, but in fact one has to subtract also a certain finite term.

\begin{proposition}
The generalized integral at $\alpha=-m$ is given by
\begin{equation}
 \gen\int_0^\infty f(r,-m)\dd r = \operatornamewithlimits{fp}_{\alpha \to -m} \gen \int f(r,\alpha)\dd r -f'_{m-1}(-m). 
 \label{eq:dimreg_anomalous}
\end{equation}
\end{proposition}
\begin{proof}
By definition, the generalized integral of $f(\cdot,-m)$ is given by \eqref{anla} evaluated at $-m$, with the term $n=m-1$ in the summation omitted. Therefore,
\begin{equation}
    \gen \int_0^\infty f(r,-m) \dd r = \lim_{\alpha \to -m} \left( \gen \int_0^\infty f(r , \alpha) \dd r - \frac{f_{m-1}(\alpha)}{\alpha+m} \right).
\end{equation}
The formula \eqref{eq:dimreg_anomalous} immediately follows.
\end{proof}

The following facts deserve emphasis.
\begin{itemize}
    \item The term $f_{m-1}'(-m)$ in \eqref{eq:dimreg_anomalous} may be nonzero even if the generalized integral of $f(r,\alpha)$ has a removable singularity at $\alpha =-m$.
    \item The finite part for $\alpha \to -m$ of the analytic continuation of $\int_0^\infty f(r, \alpha) \dd r$ is not uniquely determined by the function $f(r ,-m)$, in contrast to the generalized integral of $f(r,-m)$.
\end{itemize}

\subsection{Examples}
Let us give a few examples of generalized integrals.
\begin{example}
\begin{equation}
\gen \int_0^1\alpha r^{-1+\alpha}\dd r =\begin{cases}1 &\qquad
  \text{for } \alpha \neq 0;\\0&\qquad \text{for } \alpha=0.\end{cases}
\end{equation}
Therefore, the
the limit of generalized integrals of $f(r, \alpha) = \alpha r^{-1 + \alpha} \one_{[0,1]}(r)$
for $\alpha \to 0$ is not the
generalized integral of $f(r,0)$.
The finite part for $\alpha \to 0$ (in this case, the limit) is nonzero even though $f(r,0)=0$.
\end{example}

\begin{example}[Gamma integral]
\begin{equation*}
\gen \int_0^\infty\e^{-r} r^{-1+\alpha}\dd r =\begin{cases}\Gamma(\alpha), &\quad
\alpha\notin-\nn_0;\\[2ex]\frac{(-1)^m}{m!}\psi(m+1),
&\quad  \alpha=-m\in-\nn_0.\end{cases}
\end{equation*}
\end{example}
\begin{example}[Beta integrals] Assume that $\re(v)>0$. Then
\begin{equation*}
\gen \int_0^1 r^{-1+u}(1-r)^{v-1}\dd r =\begin{cases}\frac{\Gamma(u)\Gamma(v)}{\Gamma(u+v)}, &
 u\notin-\nn_0;\\[2ex]\frac{(-1)^m(v-m)_m\big(\psi(m+1)-\psi(v-m)\big)}{m!},&  u=-m\in-\nn_0.\end{cases}
\end{equation*}
\begin{equation*}
\gen \int_0^\infty r^{-1+u}(1+r)^{v-1}\dd r =\begin{cases}\frac{\Gamma(u)\Gamma(1-u-v)}{\Gamma(1-v)}, &
 u\notin-\nn_0;\\[2ex]\frac{(-1)^m(1-v)_m\big(\psi(m+1)-\psi(1+m-v)\big)}{m!},&  u=-m\in-\nn_0.\end{cases}
\end{equation*}
\end{example}

\section{Bessel equation}
\label{sec:bessel}

\subsection{Modified Bessel equation} 

Here is the \emph{modified Bessel equation}:
\begin{align}
\left(\partial_r^2+\frac{1}{r}\partial_r-\frac{\alpha ^2}{r^2}-1\right)v(r)& = 0.
\label{lap5}\end{align}

There are two standard solutions of (\ref{lap5}).
The first is the {\em modified Bessel function}, which can be defined by the power series
\[I_\alpha (r)=
 \sum_{n=0}^\infty\frac{\left(\frac{r}{2}\right)^{2n+\alpha }}
 {n!\Gamma(\alpha +n+1)}.\]

 The second solution is the {\em Macdonald function}, which for $\re (r)>0$ and all $\alpha $ can be defined by the absolutely convergent integral 
 \begin{align}
K_\alpha (r)&:=\frac12\int_0^\infty\exp 
\left(-\frac{r}{2}(s+s^{-1})\right)s^{ \alpha -1}\dd s.\label{basset}
 \end{align}
We have the identities
\begin{align}
K_{-\alpha }(r)=
K_\alpha (r)&=\frac{\pi}{2\sin\pi \alpha }(I_{-\alpha }(r)-I_\alpha (r)),\label{macdo1}\\
I_\alpha (r)&=\frac{1}{\ii\pi}\bigl(
K_\alpha (\e^{-\ii\pi }r)-\e^{\ii\pi \alpha }K_\alpha (r)\bigr), \label{macdo2}
\end{align}
the asymptotics
 for $|\arg r|<\pi-\epsilon$, $\epsilon>0$,
\begin{align}\lim_{|r|\to\infty} \Big( \sqrt{\frac{\pi}{2r}} \e^{-r} \Big)^{-1} K_\alpha (r)
=1,\label{asim}\end{align}
and the recurrence relations:
\begin{align}
\left(\frac{1}{r}\partial_r\right)^nr^{\pm \alpha } I_\alpha (r)&=r^{\pm \alpha -n}I_{\alpha  \mp n}(r), \\
\left(-\frac{1}{r}\partial_r\right)^n r^{\pm \alpha } K_\alpha (r)&=r^{\pm \alpha -n}K_{\alpha  \mp n}(r).
\label{eq:recur_rel_BesselK}
\end{align}
We note also the inequality \cite{Laforgia}
\begin{equation}
\frac{K_{\alpha}(r)}{K_{\alpha}(R)} < \e^{R-r} \Big( \frac{R}{r} \Big)^\alpha, \qquad \alpha > \frac12, \  0 < r < R.
\label{neat_inequality}
\end{equation}

To see \eqref{neat_inequality}, first note that
\eqref{basset} can be transformed into
 \begin{align}
K_\alpha (r)&:=\frac12\int_1^\infty\exp 
\left(-\frac{r}{2}(s+s^{-1})\right)(s^\alpha+s^{-\alpha})s^{ -1}\dd s,\label{basset1}
 \end{align}
 which easily implies
 \begin{equation}\label{basset2}
   K_\alpha(r)< K_{\alpha'}(r),\quad 0\leq\alpha<\alpha'.\end{equation}
By \eqref{eq:recur_rel_BesselK}
\begin{equation}
\frac{K_\alpha'(r)}{K_\alpha(r)}=-\frac{K_{\alpha-1}(r)}{K_\alpha(r)}-\frac{\alpha}{r}. 
  \end{equation}
Hence, for $\alpha\geq\frac12$,
using \eqref{basset2} we obtain
\begin{equation} \label{pala}
\frac{K_\alpha'(r)}{K_\alpha(r)}>-1-\frac{\alpha}{r}. 
  \end{equation}
Integrating \eqref{pala} from $r$ to $R$ we obtain 
\eqref{neat_inequality}.

\subsection{Degenerate case}

The case $\alpha\in\zz$
is called degenerate, because then the standard solutions characterized by
asymptotics at $1$ coincide.
Using for \eqref{provo} the de l'H\^opital rule, one finds for $\alpha \in \{ 0,1,2,\dots \}$
\begin{align}
  I_{\pm\alpha} (r)& = \sum_{n=0}^\infty\frac{\left(\frac{r}{2}\right)^{2n+\alpha }}
     {n!(n+\alpha )!},\\\label{provo}
K_{\pm\alpha} (r)
&=\frac12 \sum_{k=0}^{\alpha -1}
\frac{(-1)^k(\alpha -k-1)!}{k!}\Big(\frac{r}{2}\Big)^{2k-\alpha }
\\
&+\frac{(-1)^\alpha }{2}\sum_{j=0}^\infty
\frac{\big(H_j+H_{\alpha +j}-2\gamma_\mathrm{E}-2\ln(\frac{r}{2})\big)}{j!(\alpha +j)!}\Big(\frac{r}{2}\Big)^{2j+\alpha }.\notag\end{align}

\subsection{Half-integer case}
\label{ssc:half_integer_Bessel}
In the half-integer case reduce to elementary functions. More precisely, for $k=0,1,2,\dots$, we have
\begin{align} 
I_{\frac12+k}(r)&=\Big(\frac2\pi\Big)^{\frac12}r^{\frac12+k}\Big(\frac1r\partial_r\Big)^k\frac{\sinh
                  r}{r},
  \\  
  I_{-\frac12-k}(r)&=\Big(\frac2\pi\Big)^{\frac12}r^{\frac12+k}\Big(\frac1r\partial_r\Big)^k\frac{\cosh r}{r},
  \\       \label{eq:MacDonald_halfinteger}
  K_{\pm(\frac12+k)}(r)&=\Big(\frac\pi2\Big)^{\frac12}r^{\frac12+k}\Big(-\frac1r\partial_r\Big)^k\frac{\e^{- r}}{r}.
\end{align}

\subsection{Bilinear integrals}
\label{Bilinear integrals-macdonald}

First let
us describe some integral identities satisfied by Bessel functions involving the usual integral. For $\re(a+b)>0$ and $|\re(\alpha)| <1$ we have
\begin{align}
\label{integral1}
\int_0^\infty  K_\alpha (ar) K_\alpha (br)2r\dd r
&=\frac{\pi}{\sin(\pi \alpha )} \frac{(a/b)^{\alpha }-(b/a)^{\alpha }}{a^2-b^2}, \qquad \alpha \neq 0, \\
\label{integral2}
 \int_0^\infty  K_0(ar) K_0(br)2r\dd r
&=\frac{2 \ln \frac{a}{b}}{a^2-b^2}.
\end{align}
If $b=a$ with $\re(a) >0$, the above integrals reduce to
\begin{align}
    \label{int3} \int_0^\infty  K_\alpha (ar)^22r\dd r&=\frac{\pi \alpha }
{a^2\sin (\pi \alpha ) }, \qquad \alpha \neq 0,  \\
\label{int4} \int_0^\infty  K_0(ar)^22r\dd
r&=\frac{1}{a^2},\ \ \re (a)>0.
\end{align}

Formula \eqref{integral1} is 
well-known \cite[Chapter 6.521]{GR}.
It can be
obtained from the identity
\begin{align*}&
(a^2-b^2)K_\alpha (ar) K_\alpha (br)\\ =&  (\partial_r^2 + r^{-1} \partial_r )K_\alpha (ar) K_\alpha (br) -  K_\alpha (ar) (\partial_r^2 + r^{-1} \partial_r ) K_\alpha (br),
\end{align*}
which follows directly from Bessel's equation and integration by parts.
The remaining formulas are then obtained by applying the
\hospital{} rule.

If one replaces the integration in (\ref{integral1}--\ref{int4}) by
generalized integration, it makes sense for every $\alpha$. If we
treat $r^2$ as the variable of integration, these generalized integrals are non-anomalous for $\alpha \in \zz \setminus \{ 0 \}$, so~expressions stated in (\ref{integral1}, \ref{int3}) for $\alpha \neq 0$ are true for all $\alpha \not \in \zz$. Below we calculate the anomalous integrals for $\alpha \in \zz$.

\begin{proposition}
\label{prop:gen_int_mcdonald}
Let $\alpha \in \zz$ and let $\re(a+b)>0$. If $a \neq b$, then 
\begin{align}
 \label{integral2a.}
 \gen\int_0^\infty  K_\alpha (ar) K_\alpha (br)2r\dd r
&=(-1)^\alpha 2\frac{\big(\frac{a}{b}\big)^\alpha\ln\big(\frac{a}{2}\big) -\big(\frac{b}{a}\big)^\alpha\ln\big(\frac{b}{2}\big) }{a^2-b^2}\\
  -\frac{(-1)^\alpha }{ab} \sum_{k=0}^{|\alpha |-1} \Big(\frac{a}{b}\Big)^{2k - |\alpha |+1} 
 &\big(\psi(1+k)+\psi(|\alpha |-k) \big). \notag 
     \end{align}
In the case $a=b$, one has 
\begin{align}
\label{int4a.} &\gen\int_0^\infty  K_\alpha (br)^22r\dd
r =\frac{(-1)^\alpha }{b^2} \Big(1+  |\alpha
     |\ln\big(\tfrac{b^2}{4}\big)
     +2|\alpha|\big(1-\psi(1+|\alpha|)\big)\Big).
\end{align}
\end{proposition}

\begin{proof}
Both sides of  \eqref{integral2a.} are invariant with respect to
the sign flip of $\alpha$, so it is enough to consider
$\alpha=-m$ with $m\in\nn$.
We will use dimensional regularization. 

We change integration variables to $r^2$ (using Proposition
\ref{prop:genInts_scaling}). Let
\begin{align} f(r, \alpha):=\Big(\frac{ab}{4}\Big)^{-\alpha} K_\alpha(ar) K_\alpha(br).\end{align}
The generalized integral of $f(\cdot, \alpha)$ is non-anomalous for non-integer $\alpha$, so from \eqref{integral1}
\begin{align}\label{resiu}
\gen\int_0^\infty f(r,\alpha)\dd r^2=\frac{\pi}{\sin\pi\alpha}
\frac{\big(\frac{b}{2}\big)^{-2\alpha}-\big(\frac{a}{2}\big)^{-2\alpha}}{a^2-b^2}.
\end{align}

\eqref{resiu} has a simple pole at $\alpha =-m$, with residue and finite part given by
\begin{align}
\lim_{\alpha \to -m} (\alpha+m) \gen \int_0^\infty f(r,\alpha)\dd
  r^2 &= (-1)^m \frac{\big(\frac{b}{2}\big)^{2m}-\big(\frac{a}{2}\big)^{2m}}{a^2-b^2}, \\
\fp_{\alpha \to -m} \gen \int_0^\infty f(r,\alpha)\dd
  r^2 &= (-1)^m2\frac{\big(\frac{a}{2}\big)^{2m}\ln\big(\frac{a}{2}\big)
  -\big(\frac{b}{2}\big)^{2m}\ln\big(\frac{b}{2}\big)}{a^2-b^2}.
  \end{align}
By \eqref{macdo1}, all terms of $f(r, \alpha)$ singular for non-integer $\alpha$ with $\re(\alpha)<0$ are in
  \begin{align}
& \Big(\frac{ab}{4}\Big)^{-\alpha}\frac{\pi^2
                   }{4\sin^2\pi \alpha} I_{\alpha}(ar) I_{\alpha}(br)
   \label{singi} \\
    =&\frac{ 1}{4}\Bigg(\sum_{k=0}^\infty 
       \frac{\big(\frac{a}{2}\big)^{2k}r^{2k+\alpha}(-1)^k\Gamma(-\alpha-k)}{k!}\Bigg)
       \Bigg(\sum_{j=0}^\infty 
       \frac{\big(\frac{b}{2}\big)^{2j}r^{2j+\alpha}(-1)^j\Gamma(-\alpha-j)}{j!}\Bigg).
  \notag  \end{align}
The coefficient  
of \eqref{singi} at $(r^2)^{\alpha+m-1}$ is
\begin{align}
f_{m-1}(\alpha)=\frac{(-1)^{m-1}}{4}
\sum_{k=0}^{m-1}\frac{\big(\frac{a}{2}\big)^{2k}
  \big(\frac{b}{2}\big)^{2m-2-2k}\Gamma(-\alpha-k)\Gamma(-\alpha-m+1+k)}{k!(m-1-k)!}.
\end{align}
We calculate its derivative
\begin{align}
f_{m-1}'(-m)=\frac{(-1)^{m-1}}{4}
\sum_{k=0}^{m-1}\Big(\frac{a}{2}\Big)^{2k}
  \Big(\frac{b}{2}\Big)^{2m-2-2k}\big(-\psi(1+k)-\psi(m-k)\big).
  \end{align}
Formula \eqref{integral2a.} follows from
\eqref{eq:dimreg_anomalous}. It is easy to see that the generalized
integral of $K_\alpha (ar) K_\alpha (br)2r$ is a holomorphic function
of $b$, even in the anomalous case. Therefore, \eqref{int4a.} may be
obtained from \eqref{integral2a.} using the \hospital{} rule and 
identity \eqref{induction}.
\end{proof}

  \subsection{Poisson-type integral representations}

There exist two basic kinds of integral representations of solutions
of the Bessel equation. Representations similar to or derived from
\eqref{basset} are sometimes called  {\em Bessel-Schl\"afli--type}.
There exists also another family of integral representations, sometimes
called {\em Poisson-type}:
  \begin{align}
I_\alpha (r)&=\frac{1}{\sqrt{\pi}\Gamma\Big(\alpha +{\frac12}\Big)}\Big(\frac{r}{2}\Big)^\alpha 
\int_{-1}^1
(1-t^2)^{\alpha -{\frac12}}
\e^{\pm rt}\dd t, \ \  \re(\alpha )>-\frac12, \label{poi} \\
         K_\alpha (r)
  &=
  \frac{\sqrt{\pi} \left(\frac{r}{2}\right)^{\alpha }}{\Gamma(\alpha +\frac12)}\int_1^\infty\e^{-sr}(s^2-1)^{\alpha -\frac12}\dd
  s, \ \ \re(\alpha )>-\frac12, \ \re(r) > 0, \label{poiss1} \\
  K_\alpha (r) &=\frac{\Gamma(\alpha +\frac12)}{  2\sqrt\pi} \Big(\frac{r}{2}\Big)^{-\alpha }
  \int_{-\infty}^\infty\e^{-\ii sr}(s^2+1)^{-\alpha -\frac12}\dd s,\ \  \re(\alpha )>0, \   r> 0. \label{poiss2} 
\end{align}

\subsection{Standard Bessel equation}

The \emph{(standard) Bessel equation} is obtained by setting
$r\to\pm\ii r$ in the modified one:
\begin{align}
\left(\partial_r^2+\frac{1}{r}\partial_r-\frac{\alpha ^2}{r^2}+1\right)v(r)& = 0.
\label{lap6}
\end{align}

We regard the modified Bessel equation as more basic than the standard Bessel equation because its standard solutions holomorphic on $\cc \setminus ]- \infty,0]$ have simpler properties than those of the standard Bessel equation. The latter is useful mostly with $r$ restricted to $\rr_+$.

We have several kinds of standard solutions of (\ref{lap6}). The most important is 
the {\em Bessel function}, defined as
\begin{align}
J_\alpha (r)&=\e^{\pm\ii\pi\frac{ \alpha }{2}}I_\alpha (\mp\ii r).
\end{align}
The two {\em Hankel functions} also solve (\ref{lap6}):
\begin{align}\label{hanklo}
H_\alpha ^{\pm}(r)&=\frac{2}{\pi}\e^{\mp\ii\frac\pi2(\alpha +1)}
K_\alpha (\mp
\ii r).
\end{align}
We have an integral representation
\begin{align}
H_\alpha ^\pm(r)
&=\pm\frac{1}{\pi \ii}  \e^{\mp\ii\frac{\pi}{2}\alpha }\int_0^\infty
 \exp\left(\pm\ii\frac{r}{2}(s+s^{-1})\right)
s^{\alpha -1} \dd s, \ \ \pm \im(r) >0.
\label{hankel}
\end{align}

\begin{remark}
In the literature the usual  notation for Hankel functions is
\begin{align}
H_\alpha ^{(1)}(r)=H_\alpha ^+(r),\quad
H_\alpha ^{(2)}(r)=H_\alpha ^-(r).\end{align}
\end{remark}

Note the identities
\begin{align}
J_\alpha (r)&=\frac{1}{2}\left(H_\alpha ^{+}(r)+H_\alpha ^{-}(r)\right),\label{ours2}\\
J_{-\alpha }(r)&
=\frac{1}{2}\left(\e^{\alpha \pi \ii}H_\alpha ^{+}(r)+
\e^{-\alpha \pi \ii}H_\alpha ^{-}(r)\right),\\
\e^{\mp \alpha \pi \ii}H_{-\alpha }^{\pm}(r)  =
H_\alpha ^{\pm}(r)&=
\pm\frac{\ii\e^{\mp \alpha \pi \ii}J_\alpha (r)-\ii J_{-\alpha }(r)}
{\sin \alpha \pi}.
\end{align}

We have recurrence relations
\begin{align}
\left( \pm \frac{1}{r}\partial_r \right)^n r^{\pm \alpha } J_\alpha (r)
&=r^{ \pm \alpha -n}J_{\alpha  \mp n}(r).
  \end{align}
In the above recurrence relations one may replace $J_\alpha (r)$ with $H_\alpha ^\pm(r)$.

\subsection{Special values of parameters}

For $\alpha =0,1,2,\dots,$ we have
\begin{align}
  (-1)^\alpha J_{-\alpha }(r)=J_\alpha (r)&=  \sum_{n=0}^\infty\frac{(-1)^n\left(\frac{r}{2}\right)^{2n+\alpha }}
  {n!(n+\alpha )!},
\end{align}
\begin{align}
 (-1)^\alpha H_{-\alpha }^\pm(r)& = H_\alpha ^\pm(r)=
 \mp\frac{\ii}{\pi}\sum_{k=0}^{\alpha -1}
\Big(\frac{r}{2}\Big)^{2k-\alpha }
\frac{(\alpha -k-1)!}{k!}
  \\
&
\mp\frac{\ii}{\pi}\sum_{j=0}^\infty
\frac{(-1)^j\big(H_j+H_{\alpha +j}-2\gamma_\mathrm{E}\pm\ii\pi-2\ln(\frac{r}{2})\big)}{j!(\alpha +j)!}\Big(\frac{r}{2}\Big)^{2j+\alpha }
.\notag
\end{align}

For $k=0,1,2,\dots$ we have
\begin{align}
J_{\frac12+k}(r)&=\Big(\frac2\pi\Big)^{\frac12}r^{\frac12+k}\Big(-\frac1r\partial_r\Big)^k\frac{\sin r}{r}, \\
  J_{-\frac12-k}(r)&=\Big(\frac2\pi\Big)^{\frac12}r^{\frac12+k}\Big(\frac1r\partial_r\Big)^k\frac{\cos
                     r}{r}, \\
  H_{-\frac12-k}^\pm(r)=\pm\ii(-1)^kH_{\frac12+k}^\pm(r)&=\Big(\frac2\pi\Big)^{\frac12}r^{\frac12+k}\Big(\frac1r\partial_r\Big)^k\frac{\e^{\pm\ii r}}{r}.
\end{align}

\section{The Gegenbauer equation}

\label{sec:gegenbauer}
\subsection{Gegenbauer functions}

Here is the {\em Gegenbauer equation}:
\begin{align}
\Bigg((1-w^2)\partial_w^2-2(1+\alpha )w\partial_w
+\lambda ^2-\Big(\alpha +\frac{1}{2}\Big)^2\Bigg)f(w)=0.\label{gege0}
\end{align}
Its solutions can be expressed in terms of the Gauss hypergeometric
function $F(a,b;c;z)$. We will often use this function with the
so-called {\em Olver's normalization} 
\begin{align} 
{\bf F}(a,b;c;z):=\frac{F(a,b;c;z)}{\Gamma(c)}=\sum_{n=0}^\infty\frac{(a)_n(b)_nz^n}{\Gamma(c+n)n!}.
\end{align} 
The defining series converges only in the unit disc, but ${\bf F}(a,b;c;z)$ extends to a~holomorphic function on $\cc \backslash [1 , \infty[$ (as well as on a universal cover of $\cc \backslash \{ 0, 1 \}$, but we will not use the latter point of view here). 

In what follows complex power functions should be interpreted as their
principal branches (holomorphic on $\cc \backslash ]- \infty , 0
]$). For example $w \mapsto (1-w)^\alpha$ is holomorphic away from $[1
,\infty[$. In addition, we will frequently use the notation
\begin{equation}
 (w^2-1)^\alpha_\bullet := (w-1)^\alpha(w+1)^\alpha.  
\end{equation}
The function $(w^2-1)^\alpha_\bullet$ is holomorphic on $\cc\backslash]-\infty,1]$,
whereas $ (w^2-1)^\alpha$ is holomorphic on $\cc\backslash\big([-1,1]\cup\ii\rr\big)$.
One has $(w^2-1)^\alpha_\bullet = (w^2-1)^\alpha$ only for
$\re(w)>0$. However, $(1-w^2)^\alpha = (1-w)^\alpha (1+w)^\alpha$ for
all $w\not\in]-\infty,-1]\cup[1,\infty[$.

Standard solutions of the Gegenbauer equations are characterized by simple behavior at one of the three singular points $1, -1 , \infty$. Due to the $w\mapsto -w$ symmetry of the equation \eqref{gege0}, solutions of the second type are obtained from solutions of the first type by negating the argument. Therefore we consider 4 functions, corresponding to 2 behaviors at $1$ and 2 behaviors at $\infty$. All of them are holomorphic on $\cc \backslash ]- \infty, 1 ]$.

\begin{itemize}
\item The solution characterized by asymptotics $\sim 1$ at $1$:
\begin{alignat}{2}\label{solu1}
    S_{\alpha ,\pm \lambda }(w)&
  := F\Big(\frac12+\alpha+\lambda,\frac12+\alpha-\lambda;\alpha+1;\frac{1-w}{2}\Big)\\
 &= \left( \frac{2}{w+1} \right)^\alpha F\Big(\frac12+\lambda,\frac12-\lambda;\alpha+1;\frac{1-w}{2}\Big). \label{solu1_form2}
\end{alignat}
$S_{\alpha, \lambda}$ is distinguished among the four solutions
introduced here by the fact that it is holomorphic on $\cc \backslash
] - \infty, -1 ] $ rather than only on $\cc \backslash ]- \infty, 1
]$.
The equality of \eqref{solu1} and \eqref{solu1_form2} is a
  special case of the well-known {\em Euler identity}.
%{\color{blue}Equality of \eqref{solu1} and \eqref{solu1_form2} may be
%shown by comparing power series for $|1-w|<2$. A similar argument
%shows that}
On the right half-plane we have an alternative expression  obtained by
a quadratic transformation:
\begin{align}
S_{\alpha, \lambda}(w) = F\Big(\frac14+\frac\alpha2+\frac\lambda2,
  \frac14+\frac\alpha2-\frac\lambda2;\alpha+1
 ;1-w^2\Big), \qquad \re(w) > 0 .
\end{align}
\item The solution $\frac{2^{2\alpha}}{(w^2-1)^{\alpha}_\bullet} S_{-\alpha ,\lambda }(w)$ is characterized by asymptotics $\sim \frac{2^\alpha}{(w-1)^{\alpha}}$ at $1$.
\item The solution characterized by asymptotics $\sim w^{-\frac12-\alpha-\lambda}$ at $\infty$:
\begin{alignat}{2}\label{solu3}
 Z_{\alpha ,\lambda }(w)
 &=
 ( w \pm 1)^{-\frac12-\alpha -\lambda } F
\Big(
\frac12+\lambda,\frac12+\lambda+\alpha;1+2\lambda;\frac{2}{1\pm w}\Big)\\\notag
&=
w^{-\frac12-\alpha -\lambda } F\Big(\frac14+\frac\alpha2+\frac\lambda2,
\frac34+\frac\alpha2+\frac\lambda2;1+\lambda;
\frac1{w^2}\Big).
\end{alignat}
\item The solution $Z_{\alpha, - \lambda}(w)$ is characterized by asymptotics $\sim w^{-\frac12-\alpha+\lambda}$ at $\infty$.

\end{itemize}

All these 4 functions can be expressed in terms of $
S_{\alpha,\lambda}$, but for typographical reasons it is
convenient to introduce also $ Z_{\alpha,\lambda}$. We will often use
Olver's normalization:
\begin{align}
{\bf
  S}_{\alpha,\lambda}(w):=\frac{1}{\Gamma(\alpha+1)}S_{\alpha,\lambda}(w),\qquad
{\bf
  Z}_{\alpha,\lambda}(w):=\frac{1}{\Gamma(\lambda+1)}Z_{\alpha,\lambda}(w).
\end{align}
We note the identities
\begin{align}
\label{eq:identities_SZ_signs}
  {\bf S}_{\alpha,\lambda}(w)={\bf S}_{\alpha,-\lambda}(w),&\quad
     {\bf Z}_{\alpha,\lambda}(w)=\frac{ {\bf Z}_{-\alpha ,\lambda }(w)}{(w^2-1)^\alpha_\bullet}
  \end{align}
as well as the slightly more subtle {\em Whipple transformations}:
  \begin{align}
 {\bf Z}_{\alpha ,\lambda }(w)&:=  (w^2-1)^{-\frac14-\frac\alpha2 
                                -\frac\lambda2}_\bullet{\bf S}_{\lambda 
                                ,\alpha 
                                }\left(\frac{w}{(w^2-1)^{\frac12}_\bullet}\right), \label{eq:Whipple1} \\
 {\bf S}_{\alpha ,\lambda }(w)&:=  (w^2-1)^{-\frac14-\frac\alpha2 
                                -\frac\lambda2}_\bullet{\bf Z}_{\lambda 
                                ,\alpha 
                                }\left(\frac{w}{(w^2-1)^{\frac12}_\bullet}\right), \qquad \re(w) >0. \label{eq:Whipple2}
  \end{align}
    Note that $w\mapsto\frac{w}{(w^2-1)_\bullet^{\frac12}}$, defined on $\cc \setminus [-1, 1 ]$, is a holomorphic double cover of $\{ z \in \cc \setminus [-1,1] \, | \, \re(z) > 0 \}$. One has $f(-w)=f(w)$ for $w \not \in [-1,1]$ and $f(f(w))= \sgn(\re(w)) w$ for $w \not \in [-1,1 ] \cup \ii \rr$. The reason why \eqref{eq:Whipple2} (but not \eqref{eq:Whipple1}) holds only on
the right half-plane is that ${\bf Z}_{\lambda ,\alpha }(w)$ has a~branch cut on $[-1,1]$, implying that the right hand side of
\eqref{eq:Whipple2} is discontinuous on $\ii \rr$.

Here are the connection formulas:
\begin{align}\label{formu2}
 & {\bf S}_{\alpha,\lambda}(-w) \\ \notag
=&-\frac{\cos(\pi\lambda)}{\sin(\pi\alpha)}
{\bf S}_{\alpha,\lambda}(w)
+\frac{2^{2\alpha}\pi}{\sin(\pi\alpha) \Gamma(\frac12+\alpha+\lambda)
\Gamma(\frac12+\alpha-\lambda)}
\frac {{\bf S}_{-\alpha,-\lambda}(w)}{(1-w^2)^{\alpha}},\\\label{formu1}
&{\bf Z}_{\alpha,\lambda}(w) \\ \notag
=&-\frac{2^{\lambda-\alpha-\frac12}\sqrt{\pi}}{\sin(\pi\alpha)\Gamma(\frac12-\alpha+\lambda)}
{\bf S}_{\alpha,\lambda}(w) 
+\frac{2^{\lambda+\alpha-\frac12}\sqrt{\pi}}{\sin(\pi\alpha)\Gamma(\frac12+\alpha+\lambda)}
\frac {{\bf S}_{-\alpha,-\lambda}(w)}{(w^2-1)^{\alpha}_\bullet}.\end{align}

 Using the connection formulas \eqref{formu2} and \eqref{formu1} 
 (and for $\alpha\in\zz$ in addition the \hospital{} rule, see also 
 \eqref{integg} and \eqref{integ}) we derive the asymptotics
of the Gegenbauer functions
  near $w=1$   for  $\re(\alpha)>0$:
  \begin{align}
\mathbf{S}_{\alpha,\lambda}(-w)\simeq&\frac{2^{\alpha}\Gamma(\alpha)}{\Gamma(\frac12+\alpha+\lambda) \Gamma(\frac12+\alpha-\lambda)(1-w)^\alpha},\\
\mathbf{Z}_{\alpha,\lambda}(w)\simeq&\frac{2^{-\frac12+\lambda}\Gamma(\alpha)}{\sqrt\pi\Gamma(\frac12+\alpha+\lambda)(w-1)^\alpha}.\end{align}

\subsection{Recurrence relations}
The Gegenbauer functions satisfy various recurrence relations, whose
proof is straightforward:
\begin{align*}
\partial_w {\bf S}_{\alpha ,\lambda }(w)
&=-\frac{1}{2}\Big(\Big(\frac12+\alpha\Big)^2 -\lambda^2 \Big)
 {\bf S}_{\alpha +1,\lambda }(w),\\
\left((1-w^2)\partial_w -2\alpha w\right)
{\bf S}_{\alpha ,\lambda }(w)
&=-2{\bf S}_{\alpha -1,\lambda }(w) ,\\[2ex]
\left((1-w^2)\partial_w -\Big(\frac12+\alpha +\lambda \Big) w\right){\bf S}_{\alpha ,\lambda }(w)
&=-\Big(\frac12+\alpha +\lambda \Big) {\bf S}_{\alpha ,\lambda +1}(w),
\\
\left((1-w^2)\partial_w -\Big(\frac12+\alpha -\lambda \Big) w\right){\bf S}_{\alpha ,\lambda }(w)
&=-\Big(\frac12+\alpha -\lambda \Big){\bf S}_{\alpha ,\lambda
    -1}(w);
\end{align*}
\begin{align*}
  \partial_w {\bf Z}_{\alpha ,\lambda }(w)&=-\Big(\frac12+\alpha +\lambda \Big)
 {\bf Z}_{\alpha +1,\lambda }(w),\\
\left((1-w^2)\partial_w -2\alpha w\right)
{\bf Z}_{\alpha ,\lambda }(w)
&=\Big(\frac12-\alpha+\lambda\Big){\bf Z}_{\alpha -1,\lambda }(w) ,\\[2ex]
\left((1-w^2)\partial_w -\Big(\frac12+\alpha +\lambda \Big) w\right){\bf Z}_{\alpha ,\lambda }(w)
&=-\frac12 \Big(\Big(\frac12+\lambda\Big)^2 -\alpha^2 \Big) {\bf Z}_{\alpha ,\lambda +1}(w),
\\
\left((1-w^2)\partial_w -\Big(\frac12+\alpha -\lambda \Big) w\right){\bf Z}_{\alpha ,\lambda }(w)
&=2{\bf Z}_{\alpha ,\lambda -1}(w).
\end{align*}
\begin{remark}
It might be interesting to note that the above recurrence relations
correspond to short roots of the Lie algebra $so(5)$, as explained in
\cite{De1,De2}.
\end{remark}

\subsection{Gegenbauer polynomials}

In the literature one can find two kinds of polynomials related to the
functions $S_{\alpha,\lambda}$. The
{\em Jacobi polynomials} with $\alpha=\beta$ are defined
as
\begin{align}
P_n^{\alpha,\alpha}(w)&:=
\frac{(\alpha+1)_n}{n!}F\Big(-n,n+2\alpha+1;1+\alpha;\frac{1-w}{2}\Big) 
\\&=
\frac{\Gamma(\alpha+1+n)}{n!}{\bf S}_{\alpha,\frac12+\alpha+n}(w). \notag
\end{align}
However, one usually prefers the {\em Gegenbauer polynomials}
\begin{align}\label{gego}
C_n^{\alpha+\frac12}(w)&:=\frac{(2\alpha+1)_n}{(\alpha+1)_n}P_n^{\alpha,\alpha}(w)\\&=
\frac{(2\alpha+1)_n}{n!}F\Big(-n,n+2\alpha+1;1+\alpha;\frac{1-w}{2}\Big) . \notag
\end{align}
Note that 
\begin{align} P_n^{\alpha,\alpha}(-w)=(-1)^n 
P_n^{\alpha,\alpha}(w),\quad
C_n^{\frac12+\alpha}(-w)=(-1)^n 
C_n^{\frac12+\alpha}(w).
\label{jacoo}
\end{align}

We have the well-known special cases:
the {\em Legendre polynomials}
\begin{align} P_n(w)=P_n^{0,0}(w)=\mathbf{S}_{0,\frac12+n}(w);\end{align}
and the {\em Chebyshev polynomials of the first and second kind}:
    \begin{align}
      T_n(w)&=\frac{n!}{(\frac12)_n}P_n^{-\frac12,-\frac12}(w)=\sqrt\pi\mathbf{S}_{-\frac12,n}(w)\\&=
      \frac12\Big(\big(w+\ii\sqrt{1-w^2}\big)^n 
    +\big(w - \ii\sqrt{1-w^2}\big)^n\Big), \notag \\
      U_n(w)&=\frac{n!}{(\frac32)_n}P_n^{\frac12,\frac12}(w)=C_n^1(w)=\frac{n+1}{2}\sqrt{\pi}\mathbf{S}_{\frac12,n+1}(w)
      \\&= \frac1{2\ii\sqrt{1-w^2}}\Big(\big(w+\ii\sqrt{1-w^2}\big)^{n+1} -\big(w-\ii\sqrt{1-w^2}\big)^{n+1}\Big). \notag
\end{align}

\subsection{Degenerate case}

The case $\alpha\in\zz$
is called degenerate, because then the standard solutions characterized by asymptotics at $1$ become linearly dependent. More precisely, for $\alpha \in \nn$ we have 
%\marginpar{{\color{blue} Added the LHS of \eqref{eq:S_relation_degenerate2}, which got lost somewhere }}
\begin{align}
\label{eq:S_relation_degenerate2}
 {\bf S}_{\alpha ,\lambda}(w) &=\frac{2^{2\alpha}}{(\tfrac12 + \lambda)_{ \alpha}(\tfrac12 -\lambda)_{ \alpha} (1-w^2)^{\alpha}}{\bf S}_{-\alpha ,\pm\lambda}(w)
\end{align}

The functions ${\bf S}_{\alpha,\lambda}(-w)$ and  ${\bf Z}_{\alpha,\lambda}(w)$ with $\alpha \in \nn$ have the following expansions in the disc $|w-1| < 2$, involving a logarithmic singularity near~$1$: 
\begin{align}\label{integg}
  &\Gamma (\tfrac12 + \lambda + \alpha) \Gamma (\tfrac12 - \lambda + \alpha)
    \mathbf{S}_{\alpha, \lambda} (-w)\\\notag
= & \left( \frac{1-w}{2} \right)^{-\alpha}\sum_{k=0}^{\alpha-1} \frac{(\frac12 +\lambda - k)_{2k} (\alpha-k-1)!}{k! } \left( \frac{1-w}{2} \right)^k \\
&+ \sum_{j=0}^\infty \frac{(\frac12 + \lambda-\alpha-j)_{2\alpha+2j} }{j! (j+\alpha)! } \left( \frac{w-1}{2} \right)^j 
\Big( \psi(1 + \alpha + j) +\psi(1+j) 
\notag \\ 
&- \psi(\tfrac12 + \lambda + \alpha+j) 
- \psi(\tfrac12 -\lambda +\alpha + j) 
- \ln \Big(\frac{1-w}{2}\Big) \Big). \notag
\end{align}
\begin{align}\label{integ}
&\sqrt{2\pi} \, (-1)^{\alpha}2^{\alpha - \lambda} \Gamma (\tfrac12 + \lambda + \alpha)\mathbf{Z}_{\alpha, \lambda} (w)\\\notag
= & \left( \frac{1-w}{2} \right)^{-\alpha}\sum_{k=0}^{\alpha-1} \frac{(\frac12 +\lambda - k)_{2k} (\alpha-k-1)!}{k! } \left( \frac{1-w}{2} \right)^k \\
&+ \sum_{j=0}^\infty \frac{(\frac12 + \lambda-\alpha-j)_{2\alpha+2j} }{j! (j+\alpha)! } \left( \frac{w-1}{2} \right)^j 
\Big( \psi(1 + \alpha + j) +\psi(1+j) 
\notag \\ 
&- \psi(\tfrac12 + \lambda + \alpha+j) 
- \psi(\tfrac12 +\lambda -\alpha - j) 
- \ln \Big(\frac{w-1}{2}\Big) \Big). \notag
\end{align}
Let us sketch a proof of
\eqref{integ}. Using the connection formula \eqref{formu1} we can write
\begin{align*}
&\sqrt{2\pi} \, 2^{\alpha - \lambda} \Gamma (\tfrac12 + \lambda + \alpha)\mathbf{Z}_{\alpha, \lambda} (w) 
\\=&\frac{\pi}{\sin(\pi\alpha)}\Bigg(-\frac{\Gamma(\frac12+\alpha+\lambda)}{\Gamma(\frac12-\alpha+\lambda)}{\bf F}\Big(\tfrac12+\alpha+\lambda,\tfrac12+\alpha-\lambda;\alpha+1;\frac{1-w}{2}\Big)\\
&+\frac{2^{2\alpha}}{(w^2-1)_\bullet^\alpha}\Big(\frac{2}{w+1}\Big)^{-\alpha} {\bf F}\Big(\tfrac12+\lambda,\tfrac12-\lambda;-\alpha+1;\frac{1-w}{2}\Big) \Bigg) \\
 = &\frac{\pi}{\sin(\pi\alpha)}\Bigg(-\sum_{j=0}^\infty\frac{(-1)^j\Gamma(\frac12+\alpha+
    \lambda+j)}{\Gamma(\frac12-\alpha+\lambda-j)\Gamma(\alpha+1+j)j!}\Big(\frac{1-w}{2}\Big)^j\\
&  +\Big(\frac{2}{w-1}\Big)^\alpha \sum_{k=0}^\infty\frac{(\frac12+\lambda)_k (\frac12-\lambda)_k}{\Gamma(-\alpha+1+k)k!}\Big(\frac{1-w}{2}\Big)^k\Bigg).
\end{align*}
Then we apply the \hospital{} rule. We proceed similarly for \eqref{integg}.

\subsection{Half-integer case}
\label{ssc:half_integer_Gegenbauer}

If $\alpha$ is a half-integer, then the Gegenbauer functions can be
expressed in terms of elementary functions. They are particularly
simple for $\alpha=-\frac12$ and $\alpha=\frac12$:
  \begin{align} 
  \label{eq:S_minus_half}
\sqrt\pi{\bf S}_{-\frac12,\lambda}(w)
=S_{-\frac12,\lambda}(w)
&=
\frac{\big(w+\ii\sqrt{1-w^2}\big)^\lambda
      +\big(w-\ii\sqrt{1-w^2}\big)^\lambda}{2},
\\ 
\frac{\sqrt\pi}{2}{\bf 
    S}_{\frac12,\lambda}(w)=S_{\frac12,\lambda}(w)&=\frac{\big(w+\ii\sqrt{1-w^2}\big)^\lambda 
                                                    -\big(w-\ii\sqrt{1-w^2}\big)^\lambda}{2\lambda\ii\sqrt{1-w^2}},
\\
\Gamma(1+\lambda)    \mathbf{Z}_{-\frac12,\lambda}(w)=   Z_{-\frac12,\lambda}(w)
    &=2^\lambda \big((w^2-1)_\bullet^{\frac12}+w\big)^{-\lambda},\\
\Gamma(1+\lambda)   \mathbf{ Z}_{\frac12,\lambda}(w)=  Z_{\frac12,\lambda}(w)
    &= 2^\lambda \frac{(w^2-1)_\bullet^{\frac12}+w\big)^{-\lambda}}{(w^2-1)_\bullet^{\frac12}}.    
    \end{align}

    Using the recurrence relations we can find expressions for
    Gegenbauer functions
    with $\alpha\in\frac12+\zz$. More precisely, for $n=0,1,\dots$, we have 
\begin{align}
 \mathbf{S}_{-\frac12-n,\lambda}(w)
 = &\frac{(1-w^2)^{\frac12+n}}{2\sqrt\pi(-2)^n}\partial_w^n
 \frac{\big(w+\ii\sqrt{1-w^2}\big)^\lambda 
     +\big(w-\ii\sqrt{1-w^2}\big)^\lambda}{\sqrt{1-w^2}},\\
   \mathbf{S}_{\frac12+n,\lambda}(w) 
 = &\frac{2^{n}}{\ii\sqrt\pi(\lambda-n)_{2n+1}}
 \partial_w^n 
 \frac{\big(w+\ii\sqrt{1-w^2}\big)^\lambda 
          -\big(w-\ii\sqrt{1-w^2}\big)^\lambda}{\sqrt{1-w^2}} ,
      \end{align}
      \begin{align}
       \mathbf{Z}_{-\frac12-n,\lambda}(w) 
        =&\frac{(-1)^n 2^\lambda (w^2-1)_\bullet^{\frac12+n}}{\Gamma(1+\lambda+n)}\partial_w^n 
           \frac{\big((w^2-1)_\bullet^{\frac12}+w\big)^{-\lambda}}{(w^2-1)_\bullet^{\frac12}},\\
               \mathbf{Z}_{\frac12+n,\lambda}(w) 
        =&\frac{(-1)^n 2^\lambda}{\Gamma(1+\lambda+n)}\partial_w^n 
           \frac{\big((w^2-1)_\bullet^{\frac12}+w\big)^{-\lambda}}{(w^2-1)_\bullet^{\frac12}}.
        \end{align}

  \subsection{Integral representations}
  In this subsection we describe the four basic types of integral
representations of the Gegenbauer functions ${\bf S}_{\alpha,\lambda}$ and ${\bf Z}_{\alpha,\lambda}$.
    In order to understand their form it is useful to start with
    recalling some basic facts about the hypergeometric equation.

    The hypergeometric equation has 6 standard solutions, which are
traditionally presented in {\em Kummer's table}. To represent these
solutions  in a form as symmetric as possible, instead of the usual
parameters $a,b,c$ it is convenient to use the parameters
$\alpha=c-1$, $\beta=a+b-c$, $\mu=a-b$. Various elements of Kummer's
table correspond  to permutations of $\alpha,\beta,\mu$ and
switching their signs, as explained e.g. in \cite{De1,De2}.

The hypergeometric function ${\bf F}(a,b;c;z)$
is of course one of standard solutions of the hypergeometric equation.
For $\re(1+\alpha)>|\re(\beta\mp\mu)|$ it satisfies
the following pair of identities  known already to Euler
(see e.g. \cite{De1,De2} and \cite[Eq.~(15.6.1)]{NIST}):
\begin{align}
 & \mathbf{F}\Big(\frac{1+\alpha+\beta+\mu}{2}, \frac{1+\alpha+\beta-\mu}{2};\alpha+1;z\Big)\label{integ1} \\ \notag
  = \frac{1}{\Gamma(\frac{1+\alpha+\beta\mp\mu}{2})
     \Gamma(\frac{1+\alpha-\beta\pm\mu}{2})}
& \int_1^\infty
     t^{\frac{-1-\alpha+\beta\pm\mu}{2}}(t-1)^{\frac{-1+\alpha-\beta\pm\mu}{2}}(t-z)^{\frac{-1-\alpha-\beta\mp\mu}{2}}\dd t.
\end{align}
This may be verified for $|z|<1$ by Taylor expanding $t-z$ around $z=0$ and integrating term by term. It then follows that the identity is true everywhere because both left and right hand side are holomorphic on $\cc \setminus [1,\infty[$. Note that the change of the sign of $\mu$ does not affect the left
hand side of \eqref{integ1}, but yields two distinct integral representations on the right hand side.

By choosing intervals that join 
pairs of singular points $0,1,\infty,z$
of the integrand of \eqref{integ1}
one obtains integral representations of all $6$ standard solutions
 (see e.g. \cite{De1,De2}).
% Their validity may be verfied as for \eqref{integ1}.}
Let us quote
the
pair of such  representations of another standard solution, the first for $\re(1+\mu)>|\re(\alpha-
\beta)|$, 
the second for $\re(1+\mu)>|\re(\alpha+\beta)|$:
\begin{align}
&
 z^{\frac{-1-\alpha-\beta-\mu}{2}}\mathbf{F}\Big(\frac{1+\mu+\beta+\alpha}{2}, \frac{1+\mu+\beta-\alpha}{2};\mu+1;z^{-1}\Big)\notag\\
  =& 
\frac{1}{\Gamma(\frac{1+\mu+\beta-\alpha}{2})     \Gamma(\frac{1+\mu-\beta+\alpha}{2})}
     \int_0^1
     t^{\frac{-1-\alpha+\beta+\mu}{2}}(1-t)^{\frac{-1+\alpha-\beta+\mu}{2}}(z-t)^{\frac{-1-\alpha-\beta-\mu}{2}}\dd t, \label{integ2}\\
 &  =
  \frac{1}{\Gamma(\frac{1+\mu+\beta+\alpha}{2})\Gamma(\frac{1+\mu-\beta-\alpha}{2})}
\int_z^\infty
     t^{\frac{-1-\alpha+\beta-\mu}{2}}(t-1)^{\frac{-1+\alpha-\beta-\mu}{2}}(t-z)^{\frac{-1-\alpha-\beta+\mu}{2}}\dd t. \label{integ3}
  \end{align} 

To obtain representations of the Gegenbauer functions we need to set
 $\alpha=\beta$, $\mu=2\lambda$, $t=\frac{s+1}{2}$ in the
  above formulas.
  By setting  $z=\frac{1-w}{2}$ in \eqref{integ1}, we obtain the
  following pair of representations  valid
  for $\frac12>\mp\re(\lambda)>-\frac12-\re(\alpha)$:
  \begin{align}\label{deriv1}
  \mathbf{  S}_{\alpha,\lambda}(w)&=\frac{2^{\frac12+\alpha  \mp \lambda}}{\Gamma(\frac12+\alpha\mp\lambda)\Gamma(\frac12\pm\lambda)}
                                \int_1^\infty(s^2-1)^{-\frac12\pm\lambda}(s+w)^{-\frac12-\alpha\mp\lambda}\dd s. 
\end{align}
Setting $z=\frac{1+w}{2}$ in \eqref{integ2} we obtain the identity
valid for $\re(\lambda)+\frac12>0$:
  \begin{align}
   \mathbf{ Z}_{\alpha,\lambda}(w)&=\frac{1}{\sqrt\pi\Gamma(\frac12+\lambda)}
                                \int_{-1}^1(1-s^2)^{-\frac12+\lambda}
                                ( w-s)^{-\frac12-\alpha-\lambda}\dd 
                           s,\label{deriv}\end{align}
Setting $z=\frac{1+w}{2}$ in \eqref{integ3} yields the formula 
valid for $\re(\lambda)+\frac12>|\re(\alpha)|$: 
                         \begin{align}
   \mathbf{ Z}_{\alpha,\lambda}(w)
                           &=\frac{2^{2\lambda}
                             \Gamma(\lambda+\frac12)}{\sqrt\pi\Gamma(\frac12+\lambda-\alpha)\Gamma(\frac12+\lambda+\alpha)}
                                                           \int_{w}^\infty(s^2-1)^{-\frac12-\lambda}_\bullet(s-w)^{-\frac12-\alpha+\lambda}\dd 
                           s.\label{deriv3}
\end{align}
Derivation of \eqref{deriv} and \eqref{deriv3}
uses the duplication formula \eqref{dupli}.

Inserting  $w=\frac{z}{(z^2-1)_\bullet^{\frac12}}$ 
 and using the
Whipple transformations \eqref{eq:Whipple1}  and \eqref{eq:Whipple2}
we
  derive another family of integral representations.
The first follows by inserting $s=\frac{t-z}{(z^2-1)_\bullet^{\frac12}}$
  into \eqref{deriv1} and is valid for
$\re(\lambda)+\frac{1}{2}>-\re(\alpha)>-\frac{1}{2}$:
\begin{align}
                     \mathbf{ Z}_{\alpha,\lambda}(z) &=
                      \frac{2^{\frac12+\alpha+\lambda}}{\Gamma(\frac12+\alpha+\lambda)\Gamma(\frac12-\alpha)}                    \int\limits_{(z^2-1)^{\frac12}_{\bullet}+z}^\infty 
(t^2-2tz+1)^{-\alpha -\frac12}t^{-\frac12+\alpha -\lambda }\dd t. \label{quw2}
\end{align}
 To avoid having to discuss integration contour deformations involved in the change of variables, it is convenient to check this first for $z \in [1 , \infty[$ and then use analyticity.

Second representation follows by inserting
$s=\frac{-t+z}{(z^2-1)_\bullet^{\frac12}}$ into \eqref{deriv} and holds for $\re(\alpha)+\frac{1}{2}>0$: 
\begin{align}
\mathbf{  S}_{\alpha ,\lambda }(z)  &=
-  \frac {\ii (1-z^2)^{-\alpha }}{\sqrt{\pi}\Gamma(\frac12+\alpha)}
    \int\limits_{z-\ii\sqrt{1-z^2}}^{z+\ii\sqrt{1-z^2}}
    (t^2-2tz +1)^{\alpha -\frac12}_{\bullet}  t^{-\frac12-\alpha \pm\lambda }\dd t,
\label{quw1}
\end{align}
where $(t^2-2tz +1)^{\mu}_{\bullet} = (t-z-\ii \sqrt{1-z^2})^\mu (t-z+\ii \sqrt{1-z^2})^\mu$, and the integration contour is chosen so that it passes zero from the right without encircling it and does not cross the horizontal half-lines $ z \pm \ii \sqrt{1-z^2} - \rr_+$. Such contours varying smoothly with $z$ on $\cc \setminus (]-\infty,-1] \cup [1, \infty[)$ exist because on this region $\im(z - \ii \sqrt{1-z^2}) <0$ and $\im (z+ \ii \sqrt{1-z^2}) >0$. Since $\mathbf{  S}_{\alpha ,\lambda }(z)$ is continuous on $[1, \infty[$, the right hand side of \eqref{quw1} has the same limits as $z$ approaches an element of $[1 , \infty[$ from above or below. This is not obvious directly from \eqref{quw1}. The next representation follows by inserting
  $s=\frac{t+z}{(z^2-1)_\bullet^{\frac12}}$ into \eqref{deriv3}, and is
  valid for
  $\re(\alpha)+\frac{1}{2}>|\re(\lambda)|$:
  \begin{equation}
\mathbf{  S}_{\alpha ,\lambda }(z)=   \frac {2^{2\alpha}\Gamma(\alpha+\frac12)}{\sqrt\pi\Gamma(\frac12+\alpha-\lambda)
\Gamma(\frac12+\alpha+\lambda)}
\int\limits_0^{\infty}(t^2+2tz+1)^{-\alpha-\frac12}t^{-\frac12+\alpha-\lambda }\dd t.\label{quw3}
\end{equation}

It is convenient to make an exponential change of variables in the representations above.
Substituting $z = \cosh \theta$, with $\re(\theta)>0$, $|\im(\theta)| < \pi$, and $t=\e^\phi$ in \eqref{quw2}, we obtain (with the same restrictions on $\alpha, \lambda$):
\begin{align}
\mathbf{ Z}_{\alpha,\lambda}(\cosh\theta) &=
 \frac{2^{\lambda}}{\Gamma(\frac12+\alpha+\lambda)\Gamma(\frac12-\alpha)} \int\limits_{\theta}^\infty 
\big(\cosh\phi-\cosh\theta\big)^{-\alpha -\frac12}\e^{-\lambda\phi}\dd \phi. \label{quw2a}
  \end{align}
Under the same conditions on $\theta$, combining 
\eqref{eq:identities_SZ_signs} and \eqref{quw2a} gives
\begin{align}\label{quw2aa}
 \mathbf{Z}_{\alpha,\lambda}(\cosh\theta)                                                      &= \frac{2^{\lambda} 
 }{\Gamma(\frac12-\alpha+\lambda)\Gamma(\frac12+\alpha)(\sinh\theta)^{2\alpha}}                    \int\limits_{\theta}^\infty 
(\cosh\phi-\cosh\theta)^{\alpha -\frac12}\e^{ -\lambda\phi }\dd \phi. 
\end{align}

Substituting $t=\e^{\ii\phi}$ in \eqref{quw1} we obtain for 
$0 < \re(\theta) < \pi$
\begin{align}
\mathbf{  S}_{\alpha ,\lambda }(\cos\theta)  &=
  \frac {2^{\alpha-\frac12}}{\sqrt{\pi}\Gamma(\frac12+\alpha)(\sin\theta)^{2\alpha}}
    \int\limits_{-\theta}^\theta(\cos\phi-\cos\theta)^{\alpha -\frac12} \e^{\pm \ii \lambda\phi}\dd\phi.                  \label{quw1a}
\end{align}

Substituting $t=\e^\phi$, in
  \eqref{quw3} we get for $|\re(\theta)| < \pi$
  \begin{align}
\mathbf{  S}_{\alpha ,\lambda }(\cos\theta)&=   \frac {2^{\alpha-\frac12}\Gamma(\alpha+\frac12)}{\sqrt\pi\Gamma(\frac12+\alpha-\lambda) \Gamma(\frac12+\alpha+\lambda)}
\int\limits_{-\infty}^{\infty}(\cosh\phi+\cos\theta)^{-\alpha-\frac12}\e^{-\lambda\phi}\dd \phi.&\label{quw3a}\end{align}

\subsection{Asymptotics}

From the integral representations of the Gegenbauer functions, we can 
deduce the following asymptotics:

\begin{theorem}
\label{thm:asymptotics}
Let $\alpha \geq -\frac{1}{2}$ and $\pi>\delta>0$ be fixed. Then
we have the following estimates:
\begin{enumerate}
  \item Uniformly for $\theta\in[0,\pi-\delta]$ and $\beta\to\infty$, 
\begin{align}
\label{asym2}
 \frac{(\sin\theta)^{\alpha+\tfrac12}}{2^\alpha \theta^{\alpha+\tfrac12}}
 {\bf S}_{\alpha,\pm\ii\beta}(\cos\theta) 
 &= 
  (\theta\beta)^{-\alpha }
 I_{\alpha}(\theta\beta)
 \Big(1 + \theta O \big(\beta^{-1}\big) 
   \Big).\end{align}
 \item Uniformly for $\theta\in [0,\pi-\delta]$ and  $\beta\to\infty$, 
\begin{align} \label{asym3}
\frac{\pi\e^{-\pi\beta}(\sin\theta)^{\alpha+\frac12}
}{2^\alpha\theta^{\alpha+\frac12}}
  {\bf S}_{\alpha,\pm\ii\beta}(-\cos\theta)
  &
 =(\theta\beta)^{-\alpha} K_\alpha(\beta \theta)\big(1+O(\beta^{-1})\big).  
\end{align}
\item Uniformly
for $\theta\geq0$ and 
$\lambda\to\infty$,
\begin{align}\label{asym1}
\frac{\sqrt\pi\Gamma(\tfrac12-\alpha+\lambda)(\sinh \theta)^{\alpha+\frac12}}{2^{\lambda+\frac12}  
\theta^{\alpha+\frac12}}
{\bf Z}_{\alpha,\lambda}(\cosh\theta)  
&=
(\lambda\theta)^{-\alpha} K_\alpha(\lambda \theta)\big(1+O(\lambda^{-1})\big).
\end{align}\end{enumerate}
\end{theorem}

\begin{remark}
 Similar statements about associated Legendre functions can for example be 
 found in \cite{CDT,Durand19b}. See also \cite[Chapter 12, §§12,13]{Olver} 
 for detailed asymptotics of Legendre functions including also
 non-leading terms.
\end{remark}

 In the proof of Theorem \ref{thm:asymptotics}, we will use the functions
  \begin{align}
  \sinc\theta:=\frac{\sin\theta}{\theta},\quad
  \sinhc\theta:=\frac{\sinh\theta}\theta.\end{align}
We note that $\sinc0=\sinhc0=1$, $\sinc$ is decreasing on $[0,\pi]$, $\sinhc$ 
and all its derivatives are increasing on $[0,\infty[$, and for $\theta \in [0, \pi]$ we have
\begin{equation}
0 \leq   \sinc \theta \leq 1 - \frac{\theta}{\pi}, \qquad - \frac13 \theta \leq  \sinc' \theta \leq 0. \label{bellow}
\end{equation}

\begin{proof} We note first that the limiting case $\alpha=-\tfrac12$ can be proved 
explicitly by using the representations of modified Bessel functions and 
Gegenbauer functions in the half-integer case, given in Sections 
\ref{ssc:half_integer_Bessel} and \ref{ssc:half_integer_Gegenbauer}, 
respectively. For the following, let us thus assume $\alpha>-\tfrac12$.
In the proof, $c$ will be the notation for various constants
independent of $\theta, s$.

Let us prove \eqref{asym2}. The integral representation \eqref{poi} implies 
\begin{align}
\label{eq:int_scaled_BesselI}
 (\theta\beta)^{-\alpha }
 I_{\alpha}(\theta\beta) 
 &= \frac{1}{\sqrt{2\pi}\Gamma\big(\alpha+{\frac12}\big)}
 \int_{-1}^1 \Big(\frac{1-s^2}{2}\Big)^{\alpha-{\frac12}}
\e^{-\theta\beta s}\dd s.
\end{align}
  We use the integral representation \eqref{quw1a}, the substitution 
$\phi=\theta s$ and the identity\\
$\cos(s\theta)-\cos\theta=2 \sin\big(\theta\tfrac{1-s}{2}\big) 
   \sin\big(\theta\tfrac{1+s}{2}\big)$
to find 
\begin{align}
\label{eq:int_scaled_GegenS}&
\frac{(\sin\theta)^{\alpha+\tfrac12}}{2^\alpha \theta^{\alpha+\tfrac12}}
 {\bf S}_{\alpha,\pm\ii\beta}(\cos\theta)\\ 
 =& \frac{1}{\sqrt{2\pi}\Gamma(\frac12+\alpha)}
   \int\limits_{-1}^{1}
   \Bigg(\frac{2 \sin\big(\theta\tfrac{1-s}{2}\big) 
   \sin\big(\theta\tfrac{1+s}{2}\big)
   }{\theta \sin \theta}\Bigg)^{\alpha -\frac12} \e^{-\theta\beta s }\dd s. \notag
\end{align}
Subtracting \eqref{eq:int_scaled_BesselI} from \eqref{eq:int_scaled_GegenS}, 
we obtain
\begin{align}
\label{asym2.0}
& \frac{(\sin\theta)^{\alpha+\tfrac12}}{2^\alpha \theta^{\alpha+\tfrac12}}{\bf S}_{\alpha,\pm\ii\beta}(\cos\theta)
-  (\theta\beta)^{-\alpha }
 I_{\alpha}(\theta\beta)
  \\\notag
  =
    \frac{1}{\sqrt{2\pi}\Gamma(\alpha+\frac12)}
   & \int_{-1}^1 \Big(\frac{1-s^2}{2}\Big)^{\alpha-{\frac12}}\Bigg(
\Big(
\frac{\sinc\big(\theta\tfrac{1-s}{2}\big) 
\sinc\big(\theta\tfrac{1+s}{2}\big)}{\sinc\theta} \Big)^{\alpha-\frac12}-1\Bigg)\e^{-s\beta\theta}\dd 
  s. \end{align}
Next,
\begin{align}\label{esto0}
  & \frac{\sinc\big(\theta\tfrac{1-s}{2}\big) 
   \sinc\big(\theta\tfrac{1+s}{2}\big)}{\sinc\theta}
 -1\\ 
  =&\Big(
\sinc\frac{(1-|s|)\theta}{2}-1\Big)\frac{\sinc\frac{(1+|s|)\theta}{2}}{\sinc\theta}
+\frac{\sinc\frac{(1+|s|)\theta}2-\sinc\theta}{\sinc\theta}=:I+II.
  \nonumber
\end{align}
Now
\begin{align*}
|I|\leq c\frac{\theta^2 (1-|s|)^2}{\sinc\theta},
\quad |II|\leq
c\frac{\theta^2(1-|s|)}{\sinc\theta},\quad\textup{hence}\quad
  |\eqref{esto0}|\leq c\frac{\theta^2(1-|s|)}{\sinc\theta}. 
\end{align*}
$II$ is estimated as follows:
\begin{align}
\Big|\sinc\frac{(1+|s|)\theta}2-\sinc\theta\Big|&\leq
\Big|\frac{(1+|s|)}{2}\theta-\theta\Big|\sup_{\phi \in [0, \theta]} |\sinc'\phi| \leq c \theta^2 (1-|s|).  \label{esto4a}
\end{align}

Hence, using
\begin{align}
|x^\kappa-1| =  \left| \kappa \int_1^x y^{\kappa -1} \dd y \right| \leq |\kappa||x-1|\max(x^{\kappa-1},1)\label{idio}\end{align}
and $\frac{1}{\sinc(c)} \leq c$, $ \frac{\sinc \big( \theta\frac{1-s}{2} \big) \sinc \big( \theta\frac{s+1}{2} \big)}{\sinc\theta} \geq c >0$ on $[0, \pi - \delta]$, we obtain 
\begin{align}
\Bigg|\Big(  \frac{\sinc \big( \theta\frac{1-s}{2} \big) \sinc \big( \theta\frac{s+1}{2} \big)}{\sinc\theta}               \Big)^{\alpha-\frac12}-1\Bigg|
  \leq&c\theta^2(1-|s|).\label{esto4}
\end{align}
Thus, applying \eqref{esto4} and then integrating by parts, we obtain
\begin{align}\label{argu1}
  |\eqref{asym2.0}|&\leq
c\theta^2\int_{-1}^1
                     \Big(\frac{1-s^2}{2}\Big)^{\alpha-{\frac12}}(1-|s|)
                     \e^{s\beta\theta}\dd s
                     \\ \notag
  &= \frac{c \theta }{\beta}  \int_{-1}^1
    \Big(\frac{1-s^2}{2}\Big)^{\alpha-{\frac12}}
    \,\frac{2\alpha s+\sgn(s)}{1+|s|} \,\e^{s\beta\theta}\dd s
    \\ \notag
  &\leq \frac{c\theta}{\beta} (\beta\theta)^{-\alpha} I_\alpha(\beta\theta).
  \end{align}
This proves \eqref{asym2}. 

We next prove \eqref{asym3}.  
 The integral representation \eqref{poiss1} gives 
\begin{align}
\label{eq:int_scaled_BesselK}
(\theta\beta)^{-\alpha} K_\alpha(\theta\beta ) 
 =  \frac{\sqrt{\pi}}{\sqrt{2}\Gamma(\tfrac12+\alpha)}
 \int_{1}^\infty \Big(\frac{s^2-1}{2}\Big)^{\alpha-\tfrac12} 
 \e^{-\theta\beta s} \dd s. 
\end{align}
The integral representation \eqref{quw1a}, the substitution 
$\phi=\theta s-\pi$ and the identity
$\cos\theta-\cos(s\theta)=2 \sin\big(\theta\tfrac{s-1}{2}\big) 
   \sin\big(\theta\tfrac{s+1}{2}\big)$ 
yield 
\begin{align}
\label{eq:int_scaled_GegenS_2}
&\frac{\pi\e^{-\pi\beta}(\sin\theta)^{\alpha+\frac12}
}{2^\alpha\theta^{\alpha+\frac12}}
  {\bf S}_{\alpha,\pm\ii\beta}(-\cos\theta)
  \\ \notag
=&  \frac{\sqrt{\pi}}{\sqrt{2}\Gamma(\tfrac12+\alpha)}
 \int_{1}^{\tfrac{2\pi}{\theta}-1}
 \Bigg(\frac{2 \sin\big(\theta\tfrac{s-1}{2}\big) 
   \sin\big(\theta\tfrac{1+s}{2}\big)
   }{\theta \sin \theta}\Bigg)^{\alpha -\frac12}
   \e^{-\theta\beta s} \dd s . 
\end{align}
Subtracting \eqref{eq:int_scaled_BesselK} from
\eqref{eq:int_scaled_GegenS_2} we obtain 
\begin{align}
 \label{asym3.}
& \frac{\pi(\sin\theta)^{\alpha+\tfrac12}\e^{-\beta\pi}}{2^\alpha \theta^{\alpha+\tfrac12}}
 {\bf S}_{\alpha,\pm\ii\beta}(-\cos\theta) -(\beta\theta)^{-\alpha}
K_\alpha(\beta \theta)\\
  =
    \frac{\sqrt{\pi}}{\sqrt{2}\Gamma(\alpha+\frac12)}
   & \int_{1}^{\frac{2\pi}{\theta}-1} \Big(\frac{s^2-1}{2}\Big)^{\alpha-{\frac12}}\Bigg(\Big(
\frac{\sinc\theta\frac{s-1}{2}\sinc\theta\frac{s+1}{2}}{\sinc\theta}
\Big)^{\alpha-\frac12}-1\Bigg) \e^{-s\beta\theta}\dd 
     s \nonumber \\
-\frac{\sqrt{\pi}}{\sqrt{2}\Gamma(\alpha+\frac12)}
& \int_{\frac{2\pi}{\theta}-1}^\infty    \Big(\frac{s^2-1}{2}\Big)^{\alpha-{\frac12}}\e^{-s\beta\theta}\dd s  \ =:\ \frac{\sqrt{\pi}}{\sqrt{2}\Gamma(\alpha + \tfrac12)} \int_1^\infty F(s)\dd s. \nonumber
\end{align}

Arguing as in the proof of \eqref{esto4a}, we obtain
\begin{align}
\Bigg| \frac{\sinc \big( \theta\frac{s-1}{2} \big) \sinc \big( \theta\frac{s+1}{2} \big)}{\sinc\theta}             -1\Bigg|
  \leq&c\theta(s-1).\label{esto4.}
\end{align}
The power of $\theta$ is worse than in \eqref{esto4},
because for $|s| > 1$ we can no longer bound $II$ by a term proportional to $\theta^2$.

 Unfortunately, due to the zero of   $\sinc\theta\tfrac{s+1}{2}$ at $s=\tfrac{2\pi}{\theta}-1$, the remaining part of the argument has to be split into 3 cases.
  
First we consider $\alpha\geq \frac32$. Using \eqref{idio} we obtain
\begin{align}
\Bigg|\Big(  \frac{\sinc \big( \theta\frac{s-1}{2} \big) \sinc \big( \theta\frac{s+1}{2} \big)}{\sinc\theta}               \Big)^{\alpha-\frac12}-1\Bigg|
  \leq&c\theta(s-1).\label{esto4a.}
\end{align}
Therefore,
\begin{align}\label{argu5}
\Big|\int_1^{\frac{2\pi}{\theta}-1}F(s)\dd s\Big|&\leq
c\theta\int_1^{\frac{2\pi}{\theta}-1}
\Big(\frac{s^2-1}{2}\Big)^{\alpha-{\frac12}}(s-1)
\e^{-s\beta\theta}\dd s.
\end{align}
Next, using $\theta(s-1)\geq2\pi-2\theta$ we obtain
\begin{align}\label{argu6}
\Big|\int_{\frac{2\pi}{\theta}-1}^\infty F(s)\dd s \Big|  &\leq
\theta(2\pi-2\theta)^{-1}\int_{\frac{2\pi}{\theta}-1}^\infty
\Big(\frac{s^2-1}{2}\Big)^{\alpha-{\frac12}}(s-1)
\e^{-s\beta\theta}\dd s
    \end{align}
Then we sum up \eqref{argu5} and \eqref{argu6}.
We choose the new $c$ to be the maximum of the old $c$ and $(2\delta)^{-1}$. Next we integrate by parts:
    \begin{align}\notag
\Big|\int_1^\infty F(s)\dd s\Big|&\leq
c\theta\int_1^{\infty}
                     \Big(\frac{s^2-1}{2}\Big)^{\alpha-{\frac12}}(s-1)
                     \e^{-s\beta\theta}\dd s\\\notag
  &= \frac{c  }{\beta}  \int_1^\infty
    \Big(\frac{s^2-1}{2}\Big)^{\alpha-{\frac12}}
    \frac{2\alpha s+1}{1+s} \e^{-s\beta\theta}\dd s
  \\\label{argu4}
  &\leq \frac{c }{\beta} (\beta\theta)^{-\alpha}
K_\alpha(\beta\theta).
\end{align}

Next we consider $\frac12\leq \alpha\leq\frac32$.
Choose $\sigma \in ] \pi,\pi+\delta[$. For $1\leq
s\leq\frac{\sigma}{\theta}$, we have
\begin{align}\pi-\theta\frac{s+1}{2}\geq\frac{\pi-\sigma+\delta}{2}>0,
\end{align}
so for such $s$ the argument of $\sinc \big(\theta\frac{s+1}{2} \big)$ is separated from the location of the zero by a constant independent of $\theta$ and $s$. Hence we have a bound
\begin{align}
\frac{\sinc \big( \theta\frac{s-1}{2} \big) \sinc \big( \theta\frac{s+1}{2} \big)}{\sinc\theta}
  \geq c>0.\end{align}
  Therefore, we can apply \eqref{idio} and obtain \eqref{esto4a.} on
  this interval. Thus
\begin{align}\label{argu5.}
\Big|\int_1^{\frac{\sigma}{\theta}}F(s)\dd s\Big|&\leq
c\theta\int_1^{\frac{\sigma}{\theta}}
\Big(\frac{s^2-1}{2}\Big)^{\alpha-{\frac12}}(s-1)
\e^{-s\beta\theta}\dd s.
\end{align}

On $\frac{\sigma}{\theta}\leq s\leq \frac{2\pi}{\theta}-1$ we have
$\delta<\sigma-\theta\leq\theta(s-1)$. We estimate the left hand side of \eqref{esto4a.} by a constant and get
\begin{align}\label{argu5..}
\Big|\int_{\frac{\sigma}{\theta}}^{\frac{2\pi}{\theta}-1}F(s)\dd s\Big|&\leq
c\frac{\theta}{\delta}\int_{\frac{\sigma}{\theta}}^{\frac{2\pi}{\theta}-1}
\Big(\frac{s^2-1}{2}\Big)^{\alpha-{\frac12}}(s-1)
                     \e^{-s\beta\theta}\dd s.
\end{align}
Now combine \eqref{argu5.}, \eqref{argu5..}, and
\eqref{argu6} and argue as in \eqref{argu4}.

Finally, consider $-\frac12 < \alpha\leq\frac12$. We need to consider only the interval $\frac\sigma{\theta}\leq s\leq
\frac{2\pi}{\theta}-1$. Contributions of the remaining part of the integration region are $O(\beta^{-1})$ by previous arguments.

  First,
  \begin{align} 
  s^2-1\geq\frac{\sigma^2-\theta^2}{\theta^2}\geq
    \frac{c}{\theta^2}\quad\Rightarrow\quad\
    \Big(\frac{s^2-1}{2}\Big)^{\alpha-{\frac12}}\leq c\theta^{1-2\alpha}.\end{align}
Moreover,
  using \eqref{bellow} we have
  \begin{align}
\frac{1}{\sin\theta}\geq c,&\quad\sinc \big( \theta\tfrac{s\pm1}{2} \big) \geq
    c\theta\big(\tfrac{2\pi}{\theta}-s\mp1\big),
\end{align}
which implies
\begin{align}
     &\Big(
\frac{\sinc \big( \theta\frac{s-1}{2} \big) \sinc \big( \theta\frac{s+1}{2} \big)}{\sinc\theta}
  \Big)^{\alpha-\frac12}-1 \leq
 \Big(
\frac{\sinc \big( \theta\frac{s-1}{2} \big) \sinc \big( \theta\frac{s+1}{2} \big)}{\sinc\theta}
                            \Big)^{\alpha-\frac12}\\
\leq &\; c    \theta^{2\alpha-1}\Big(\big(\tfrac{2\pi}{\theta}-s\big)^2-1\Big)^{\alpha-\frac12}
  =c\theta^{2\alpha-1}(t^2-1)^{\alpha-\frac12},
  \notag \end{align}
where we set $t=\frac{2\pi}{\theta}-s$.
Then we estimate (very roughly)
\begin{align}\label{argu5...}
  \Big|\int_{\frac{\sigma}{\theta}}^{\frac{2\pi}{\theta}-1}
  F(s)\dd s\Big|&\leq
c\int_{1}^{\frac{2\pi-\sigma}{\theta}}(t^2-1)^{\alpha-\frac12}\e^{-\beta 
 2\pi+\beta\theta t}  \dd t  \\
  &\leq c\int_{1}^{\frac{2\pi-\sigma}{\theta}}(t^2-1)^{\alpha-\frac12}\e^{\beta 
    2(\pi-\sigma)-\beta\theta t}   \dd t \\& \leq \ c\e^{\beta 2(\pi-\sigma)}(\beta\theta)^{-\alpha} K_\alpha(\beta\theta)=(\beta\theta)^{-\alpha} K_\alpha(\beta\theta) O( \beta^{-\infty} ).
\end{align}

Finally, let us prove \eqref{asym1}.
  Substituting $\phi=s\theta$ in \eqref{quw2aa}
and using the identity $\cosh(s\theta)-\cosh(\theta)
=2\sinh\big(\theta\tfrac{s-1}{2}\big)\sinh\big(\theta\tfrac{s+1}{2}\big)$, 
we find
\begin{align}
  \label{eq:int_scaled_GegenZ}
  &\frac{\sqrt\pi\Gamma(\tfrac12-\alpha+\lambda)(\sinh \theta)^{\alpha+\frac12}}{2^{\lambda+\frac12}  
\theta^{\alpha+\frac12}}
{\bf Z}_{\alpha,\lambda}(\cosh\theta)  
\\
=&\frac{\sqrt{\pi}}{\sqrt{2} \Gamma(\tfrac12+\alpha)}
\int\limits_{1}^\infty 
\Bigg(\frac{2\sinh\big(\theta\frac{s-1}{2}\big)\sinh\big(\theta\frac{s+1}{2}\big)}{\theta\sinh\theta}\Bigg)^{\alpha -\frac12}\e^{ -\lambda\theta s }\dd s. 
\end{align}
Subtracting \eqref{eq:int_scaled_BesselK} with $\beta$ replaced with $\lambda$ from \eqref{eq:int_scaled_GegenZ}, we obtain
\begin{align}\label{asym1.}
&\frac{\sqrt\pi\Gamma(-\alpha+\lambda+\frac12)(\sinh \theta)^{\alpha+\frac12}}{2^{\lambda+\frac12}  
\theta^{\alpha+\frac12}}
{\bf Z}_{\alpha,\lambda}(\cosh\theta) -(\lambda\theta)^{-\alpha} K_\alpha(\lambda \theta) \\\notag
= 
& c
 \int_1^{\infty} \Big(\frac{s^2-1}{2}\Big)^{\alpha-{\frac12}}\Bigg(\Big(
\frac{\sinhc \big( \theta\frac{s-1}{2} \big)\sinhc \big( \theta\frac{s+1}{2} \big)}{\sinhc\theta}
 \Big)^{\alpha-\frac12}-1\Bigg)\e^{-s\lambda\theta}\dd    s         . 
\end{align}

Function $s\mapsto\frac{\sinhc\big(\theta\frac{s-1}{2}\big)\sinhc\big(\theta\frac{s+1}{2}\big)
 }{\sinhc(\theta) }$ is monotonically increasing and equals $1$ at $s=1$.
Arguing as 
 in the proof of \eqref{esto4} we obtain
 \begin{align}
\Bigg|   \frac{\sinhc \big( \theta\frac{s-1}{2} \big)\sinhc \big( \theta\frac{s+1}{2} \big)}{\sinhc\theta}-1 \Bigg|
  \leq&c\theta(s-1)\e^{\theta(s-1)}.\label{esto4a..}
\end{align}
Next we note that
\begin{align}
   & \frac{\sinhc \big( \theta\frac{s-1}{2} \big)\sinhc \big( \theta\frac{s+1}{2} \big)}{\sinhc\theta} = \frac{2}{\theta (s^2-1)} (\e^{\theta (s-1)}-1) \frac{1- \e^{- \theta (s+1)}}{1-\e^{-2 \theta}} \\
   \notag  \leq &  \frac{2 \e^{\theta(s-1)}}{s+1} \frac{1- \e^{- \theta (s+1)}}{1-\e^{-2 \theta}} \leq \e^{\theta(s-1)},
\end{align}
where in the first inequality we used $\e^x-1\leq x \e^x$, and in the second one we optimized $\frac{1- \e^{- \theta (s+1)}}{1-\e^{-2 \theta}}$ with respect to $\theta$. Invoking \eqref{idio} again, we get
\begin{align}
\Bigg|\Big(  \frac{\sinhc\theta\frac{s-1}{2}\sinhc\theta\frac{s+1}{2}}{\sinhc\theta}               \Big)^{\alpha-\frac12}-1\Bigg|
  \leq&c\theta(s-1)\e^{\theta(s-1)\rho(\alpha)},\label{esto4b}
\end{align}
where $\rho(\alpha) = \max \{ 1 , \alpha - \tfrac12 \}$. 
Then, integrating by parts, we obtain
\begin{align}
  |\eqref{asym1.}|\leq
&
c\theta\int_1^{\infty} \Big(\frac{s^2-1}{2}\Big)^{\alpha-{\frac12}}(s-1) \e^{-s\lambda\theta+(s-1)\theta\rho(\alpha)}\dd s\\\notag
  =&\frac{c  }{\lambda-\rho(\alpha)} \e^{-\theta\rho(\alpha)} \int_1^{\infty}
    \Big(\frac{s^2-1}{2}\Big)^{\alpha-{\frac12}}
    \frac{2\alpha s+1}{1+s} \e^{-s\theta(\lambda-\rho(\alpha))}\dd s\\
\leq&\frac{c \e^{-\theta\rho(\alpha)} }{\lambda-\rho(\alpha)} 
 \big(\theta(\lambda-\rho(\alpha))\big)^{-\alpha} K_\alpha\big(\theta(\lambda-\rho(\alpha))\big) \notag \\
 \leq & \frac{c}{\lambda - \rho(\alpha)} \Big( \frac{\lambda}{\lambda-\rho(\alpha)} \Big)^{2 \alpha} (\lambda \theta)^{- \alpha} K_{\alpha}(\lambda \theta), \notag
\end{align}
where in the last step we used \eqref{neat_inequality}.
 \end{proof}

\subsection{Bilinear integrals}

In this subsection we compute certain integrals involving products of two
Gegenbauer functions. They are analogs of the integrals for products
of Macdonald functions considered in Subsect. \ref{Bilinear integrals-macdonald}.
We will use the integration variable $2w$, because this will
facilitate comparison of the integrals for Gegenbauer functions with
those for Macdonald functions. Indeed, if we consider the spherical,
resp. hyperbolic interpretation of these identities, for small $r$, we have
$2(w\mp1)\approx r^2$. Therefore, $\dd 2w$ corresponds to 
$2r\dd r=\dd r^2$. This is especially important in the case of anomalous generalized
integrals in Theorem \ref{thm:asymp_genintSZ}, where without
a good choice of variables we would not have obtained a correct limit.
    
\begin{theorem}
\label{thm:S_integrals}   
 For $|\re(\alpha)|<1$, the following identities hold:
\begin{align}\label{formo}
&\int_{-1}^1{\bf S}_{\alpha,\ii\beta_1}(w){\bf
  S}_{\alpha,\ii\beta_2}(w)(1-w^2)^\alpha\dd 2w
\\=&   \frac{2^{2 \alpha +2}}{(\beta_1^2-\beta_2^2)\sin\pi\alpha} \Big(\frac{\cosh(\pi\beta_1)
}{\Gamma(\frac12+\alpha-\ii\beta_2) 
  \Gamma(\frac12+\alpha+\ii\beta_2)}
  -(\beta_1\leftrightarrow\beta_2)\Big)\notag
\\
&\int_{-1}^1{\bf S}_{0,\ii\beta_1}(w){\bf
  S}_{0,\ii\beta_2}(w)\dd 2w &\\\notag
  &=
   \frac{4 \cosh(\pi\beta_1) \cosh(\pi\beta_2)\Big(
   \psi\big(\tfrac12-\ii\beta_1\big)
   +\psi\big(\tfrac12+\ii\beta_1\big)
   -(\beta_1\leftrightarrow\beta_2)\Big)
   }{\pi^2(\beta_1^2-\beta_2^2)},
\end{align}
\begin{align}
&\int_{-1}^1{\bf S}_{\alpha,\ii\beta}(w)^2(1-w^2)^\alpha\dd 2w
=
    \frac{ 2^{2\alpha+1} \ii \cosh(\pi\beta) }
    {\beta\sin\pi\alpha \,\Gamma(\tfrac12+\alpha-\ii\beta)
    \Gamma(\tfrac12+\alpha+\ii\beta)}
\\ \notag
    &\times \Big( \psi\big(\tfrac12 +\alpha + \ii \beta \big) 
         - \psi\big(\tfrac12 +\alpha - \ii \beta \big) 
         + \psi\big(\tfrac12 - \ii \beta \big) 
         - \psi\big(\tfrac12 + \ii \beta \big) \Big)
    ,\\
  &
    \int_{-1}^1{\bf S}_{0,\ii\beta}(w)^2\dd 2w
  =   \frac{ 2 \ii \cosh^2(\pi\beta)
    \Big(\psi'(\frac12+\ii\beta)-\psi'(\frac12-\ii\beta)\Big)}
    {\beta\pi^2},\end{align}
\begin{align}
 &\int_{-1}^1{\bf S}_{\alpha,0}(w)^2(1-w^2)^\alpha\dd 2w\ 
=\ \frac{ 2^{2\alpha+1 } \left( \pi^2-2\psi'\big(\frac12+\alpha\big) \right)}
{\sin(\pi\alpha)\Gamma(\frac12+\alpha)^2},
\\ &
\int_{-1}^1{\bf S}_{0,0}(w)^2\dd 2 w\ =\                                                 -\frac{4  \psi''\big(\frac12\big)}
{\pi^2}.
  \end{align}
  \end{theorem}

  \begin{proof} The Gegenbauer equation implies  
  \begin{align*}
    &(\beta_1^2-\beta_2^2) \int_{-1}^1 
{\bf S}_{\alpha,\ii\beta_1}(w){\bf 
      S}_{\alpha,\ii\beta_2}(w)(1-w^2)^\alpha\dd w
      \\ \notag=&
    \int_{-1}^1 
{\bf S}_{\alpha,\ii\beta_1}(w) \partial_w (1-w^2)^{\alpha +1} \partial_w
{\bf  S}_{\alpha,\ii\beta_2}(w) \dd w 
\\ \notag & -\int_{-1}^1
\Big(\partial_w (1-w^2)^{\alpha+1}  \partial_w{\bf S}_{\alpha,\ii\beta_1}(w) \Big){\bf 
                                                           S}_{\alpha,\ii\beta_2}(w)\dd w\\
    =&\lim_{w\searrow-1} \Big( (1-w^2)^{\alpha +1}\Big(
{\bf S}_{\alpha,\ii\beta_1}(w)  \partial_w{\bf 
                                                           S}_{\alpha,\ii\beta_2}(w)
                                                           -
                \big( \partial_w{\bf S}_{\alpha,\ii\beta_1}(w)\big) {\bf 
                                                           S}_{\alpha,\ii\beta_2}(w)\Big)\Big).
  \end{align*}
  Then we use the connection formula \eqref{formu2} and obtain
  \eqref{formo}.
  The remaining identities follow by the  de  l'H\^opital rule.  \end{proof}

Integrals in Theorem \ref{thm:S_integrals} are defined as generalized integrals for $\re(\alpha)>-1$ (beyond this region one would need to use a generalized integral which also takes care of non-integrability at the right endpoint). If $\alpha\notin\nn$, these generalized integrals are non-anomalous and hence formulas from 
Theorem \ref{thm:S_integrals} remain true. Below we compute anomalous generalized integrals 
for $\alpha\in\nn$.

\begin{theorem}\label{thm:genIntsSS}
  Let $\alpha\in \nn$. Then
\begin{align}
 &  (-\tfrac14)^{\alpha+1} 
 \Big( \prod_{i=1}^2 \prod_{\pm} \Gamma(\tfrac{1}{2} + \alpha \pm \ii \beta_i) \Big)   \gen\int_{-1}^1{\bf S}_{\alpha,\ii\beta_1}(w)
 {\bf S}_{\alpha,\ii\beta_2}(w) (1-w^2)^\alpha\dd 2w  \label{eq:formo_integer1.} \\ 
 \notag =  & \frac{1}{\beta_1^2 - \beta_2^2}  \Bigg( \big(\tfrac12 +\ii\beta_1 \big)_{\alpha} \big(\tfrac12 -\ii\beta_1 \big)_{\alpha}
\Big( \ln 4 - \sum_{\pm} \psi\big( \tfrac12 +\alpha \pm \ii\beta_1 \big) \Big) 
   - (\beta_1 \leftrightarrow \beta_2) \Bigg) \\
\notag & + \sum_{k=0}^{\alpha-1}  \Bigg( \prod_{\pm} \frac{(\tfrac12 \pm \ii \beta_2)_{\alpha} (\tfrac12 \pm \ii \beta_1)_k}{(\tfrac12 \pm \ii \beta_2)_{k+1}} \Bigg)  \\
 \notag & \times \Big(  \psi(\alpha-k)+\psi(1+k) -\sum_{\pm}
  H_{\alpha-1-k}\big(\tfrac{1}{2}-\alpha\pm \ii \beta_2 \big) 
\Big). 
 \end{align}

\begin{align}
  \label{eq:formo_integer1_limit1.}
  & (- \tfrac14)^{\alpha+1} \frac{\pi}{\cosh \pi \beta}  \Big( \prod_\pm \Gamma(\tfrac12 + \alpha \pm \ii \beta) \Big)  \gen \int_{-1}^1 \mathbf S_{\alpha,\ii \beta}(w)^2 (1-w^2)^{\alpha} \dd 2 w \\
 \notag  = &  \sum_{k=0}^{\alpha-1} \frac{1}{(k+\tfrac12)^2 + \beta^2} \Big( \ln 4 + \psi(\alpha-k) + \psi(k+1) - \sum_\pm \psi(-\tfrac12 - k \pm \ii \beta) \Big) \\
 \notag & - \frac{\ii}{2 \beta} \big( \psi'(\tfrac12 + \alpha + \ii \beta) - \psi'(\tfrac12 + \alpha - \ii \beta) \big).
\end{align}

\begin{align}
 \label{eq:formo_integer1_limit2.}
&  (- \tfrac14)^{\alpha+1} \pi \Gamma(\tfrac12 + \alpha)^2 \ \gen \int_{-1}^1 \mathbf S_{\alpha,0}(w)^2 (1-w^2)^{\alpha} \dd 2 w \\
 \notag = & \psi''(\tfrac12 + \alpha) + \sum_{k=0}^{\alpha-1} \frac{\ln 4 + \psi(\alpha-k) + \psi(k+1) - 2 \psi(-\tfrac12 -k)}{(k + \tfrac12)^2}
\end{align}
\end{theorem}

\begin{remark}
    The right-hand side of \eqref{eq:formo_integer1.} is not manifestly symmetric in $\beta_1,\beta_2$, although the left-hand side is. This is a nontrivial sum identity which follows from the equality of \eqref{solu1} and \eqref{solu1_form2}. The same is true for \eqref{eq:gen_intZZ.} below.
\end{remark}

\begin{proof}
Let us first compute \eqref{eq:formo_integer1.}. We will use the
dimensional regularization. 
Thus $\alpha$ at first  takes arbitrary complex values with $\re(\alpha)>-1$. We define 
\begin{align}
 f(\alpha, w):= 2^{-4\alpha} 
 \left[ \prod_{i=1}^2 \prod_{\pm} \Gamma(\tfrac12 + \alpha \pm \ii \beta_i ) \right] \mathbf{S}_{\alpha,\ii\beta_1}(w) \mathbf{S}_{\alpha,\ii\beta_2}(w)(1-w^2)^\alpha.
\end{align}

In terms of $f(\alpha, w)$, \eqref{formo} becomes (for 
$|\re(\alpha)|\leq1$) 
 \begin{align}
 \label{kaku}
  &\quad\int_{-1}^1  f(\alpha, w) \dd 2w =\frac{2^{-2 \alpha +2}\pi }{(\beta_1^2-\beta_2^2)\sin\pi\alpha}
    \Bigg( 
 \prod_{\pm} \frac{\Gamma(\tfrac12 + \alpha \pm \ii \beta_1)}{\Gamma(\tfrac12 \pm \ii \beta_1)} - (\beta_1 \leftrightarrow \beta_2) \Bigg).
  \end{align} 

Let $m\in\nn_0$. The residue and the finite part of \eqref{kaku} at $\alpha=m$ are
\begin{align}
\res(m)
 = &  \frac{(-1)^{m} 2^{-2m +2} 
    }{\beta_1^2-\beta_2^2 }
    \Big(\big(\tfrac12-\ii\beta_1\big)_m
  \big(\tfrac12+\ii\beta_1\big)_m
  -(\beta_1\leftrightarrow\beta_2)\Big),
\end{align}
\begin{align}
\label{eq:S_int_pure_finite_part}
  & (-1)^m 2^{2m-2} (\beta_1^2 - \beta_2^2)  \fp_{\alpha \to m}\int_{-1}^1f(\alpha,w)\dd 2w  \\
  =  
 \notag 
 & \big(\tfrac12-\ii\beta_1\big)_m
  \big(\tfrac12+\ii\beta_1\big)_m  
     \Big(-\ln(4) +\sum_{\pm} \psi\big( \tfrac12 + m \pm \ii\beta_1 \big)  \Big)
          - (\beta_1\leftrightarrow\beta_2).
          \end{align}

 All negative powers of $w+1$ in $f(\alpha,w)$ are contained in 
 \begin{align*}
f^\sing(\alpha,w):=  \frac{ 2^{-2\alpha} \pi^2 \big(\frac{1-w}{2}\big)^{-\alpha} 
   \big(\frac{1+w}{2}\big)^{-\alpha} 
   {\bf S}_{-\alpha,\ii\beta_1}(-w){\bf S}_{-\alpha,\ii\beta_2}(-w)
   }{\sin^2\pi\alpha}.
 \end{align*}

In order to avoid having to expand $(1-w)^\alpha$ in terms of 
$1+w$, we choose to expand one Gegenbauer function using 
\eqref{solu1} and the other one using \eqref{solu1_form2}:
\begin{align}
\label{summa} 
& (2(1+w))^{\alpha} f^\sing(\alpha,w) = \frac{   \pi^2    }{\sin^2\pi\alpha}
 {\bf 
  F}\Big(\frac12+\ii\beta_1,\frac12-\ii\beta_1,-\alpha+1,\frac{1+w}{2}\Big)\\\notag
  & \times 
 {\bf 
  F}\Big(\frac12-\alpha+\ii\beta_2,\frac12-\alpha-\ii\beta_2,-\alpha+1,\frac{1+w}{2}\Big)  
\\\notag  =  & \Bigg(\sum_{k=0}^\infty\frac{(-1)^k(\frac12+\ii\beta_1)_k 
         (\frac12-\ii\beta_1)_k\Gamma(\alpha-k)}{k!}\Big(\frac{1+w}{2}\Big)^k\Bigg)\\ \notag
     &\times   
     \Bigg(\sum_{j=0}^\infty\frac{(-1)^j(\frac12-\alpha+\ii\beta_2)_j 
           (\frac12-\alpha-\ii\beta_2)_j\Gamma(\alpha-j)}{j!}\Big(\frac{1+w}{2}\Big)^j\Bigg).
 \end{align}

The coefficient at 
$\big(2(1+w)\big)^{-\alpha+m-1}$ in $f^\sing(\alpha,w)$ is
\begin{align*}
 f_{m-1}(\alpha)=& 
   \sum_{k=0}^{m-1}
   \frac{\Gamma(\alpha-k)  \Gamma(\alpha-m+1+k)}{k!(m-1-k)! (-4)^{m-1}} \prod_{\pm} \big( \tfrac{1}{2} \pm \ii \beta_1 \big)_{k} \big( \tfrac{1}{2}-\alpha \pm \ii \beta_2 \big)_{m-1-k}.
\end{align*}
We evaluate the derivative
\begin{align}
  &f_{m-1}'(m) =
  (-\tfrac14)^{m-1} \sum_{k=0}^{m-1} \prod_{\pm} \big( \tfrac{1}{2} \pm \ii \beta_1 \big)_{k} \big( \tfrac{1}{2}-m \pm \ii \beta_2 \big)_{m-1-k}
 \\ \notag & \times \Bigg( \psi(m-k) + \psi(1+k) - \sum_{\pm} H_{m-1-k}\Big(\frac{1}{2}-m \pm \ii \beta_2 \Big)  \Big).
\end{align}
Then we use
\begin{align}
 \big( \tfrac12 -m+\ii \beta_2 \big)_{m-1-k} 
 \big( \tfrac12 -m-\ii \beta_2 \big)_{m-1-k}=\frac{(\frac12+\ii\beta_2)_m (\frac12-\ii\beta_2)_m}{(\frac12+\ii\beta_2)_{k+1}(\frac12-\ii\beta_2)_{k+1}}.
  \end{align}
Since the singular exponent is $- \alpha + m -1$, we have to \emph{add} $f_{m-1}'(m)$ to the finite part of the bilinear integral.\footnote{For $\re(\alpha)>0$, the derivative term 
 has to be added to the finite part, not subtracted: 
 \begin{align*}
  -\frac{f_{m-1}(\alpha)}{m-\alpha} 
  = -\frac{f_{m-1}(m) + (\alpha-m) f'_{m-1}(m)}{m-\alpha} + O(\alpha-m)
  =  \frac{\res(m) }{\alpha-m}+ f'_{m-1}(m) + O(\alpha-m).
 \end{align*}
}
Thus, we obtain the generalized integral \eqref{eq:formo_integer1.} 
  for positive integers by multiplying 
  \begin{align}
   \gen\int_{-1}^1f(m,w)\dd 2w 
   =&  \fp_{\alpha \to m} \int_{-1}^1f(\alpha,w)\dd 2w 
   + f_{m-1}'(m)
  \end{align}
 with 
  \begin{align}
 \frac{2^{4m}  }{
 \Gamma(\tfrac12+m+\ii\beta_1)\Gamma(\tfrac12+m-\ii\beta_1)\Gamma(\tfrac12+m+\ii\beta_2)\Gamma(\tfrac12+m-\ii\beta_2) }.
 \end{align}

  Let us now prove \eqref{eq:formo_integer1_limit1.}. Since the 
  exponents of the non-integrable terms do not depend on $\beta$, 
  it suffices to take the limit $\beta_1,\beta_2\to \beta$ of 
  \eqref{eq:formo_integer1.}. The multiplicative factor in front 
  of the left-hand side becomes
  \begin{equation}
  \label{eq:prefactorS1}
      (-\tfrac14)^{\alpha+1} \prod_{\pm} \Gamma(\tfrac12 + \alpha \pm \ii \beta)^2 
      = (-\tfrac14)^{\alpha+1} \frac{\pi}{\cosh \pi \beta} 
      \prod_{\pm} \Gamma(\tfrac12 + \alpha \pm \ii \beta) (\tfrac12 \pm \ii \beta)_{\alpha}. 
  \end{equation}

Applying the  de l'Hôpital rule to the second line of
\eqref{eq:formo_integer1.} then yields 
\begin{align}
\label{qrqr}    \frac{\ii (\tfrac12 + \ii \beta)_{\alpha} (\tfrac12 - \ii \beta)_{\alpha}}{2 \beta}  \Big( &  \big( \ln 4 - \sum_\pm \psi(\tfrac12 + \alpha \pm \ii \beta) \big) \big( \sum_\pm \pm H_\alpha(\tfrac12 \pm \ii \beta) \big)  \\
    \notag & - \sum_\pm \pm \psi'(\tfrac12 + \alpha \pm \ii \beta)   \Big)
\end{align}

The third and fourth line of \eqref{eq:formo_integer1.} can be converted to the following form
\begin{align}\label{qrqr1}
           & (\tfrac12+\ii\beta)_{\alpha}   (\tfrac12-\ii\beta)_{\alpha}
        \sum_{k=0}^{\alpha -1} \frac{1 }{\big(\tfrac12 +k\big)^2+\beta^2
  }  \\
  \times&\Big( \psi(\alpha - k) + \psi(k+1) - \sum_{\pm}  H_{\alpha -1-k}\big(\tfrac12-\alpha 
 \pm \ii\beta\big)
 \Big)
 . \notag
        \end{align}
        
Next we use the identities
\begin{align}
 \sum_{k=0}^{\alpha-1} 
 \frac{1}{\big(\tfrac12 +k)^2+\beta^2} 
 = &
 \frac{\ii}{2\beta} 
  \Big( H_{\alpha}\big(\tfrac12+\ii\beta\big)
     -H_{\alpha}\big(\tfrac12-\ii\beta\big)\Big),\\
   \psi\big( \tfrac{1}{2} + \alpha + \ii\beta \big) 
+ \psi\big( \tfrac12  + \alpha - \ii\beta \big)
  &=
\psi\big( \tfrac{1}{2} -\alpha + \ii\beta \big) 
+ \psi\big( \tfrac12  - \alpha - \ii\beta \big),
       \\
  H_{\alpha-1-k}(\tfrac12-\alpha\pm\ii\beta)
 & =\psi(-\tfrac12-k\pm\ii\beta) - \psi(\tfrac12-\alpha\pm\ii\beta),
\end{align}
to combine \eqref{qrqr1} and \eqref{qrqr}.
This proves \eqref{eq:formo_integer1_limit1.}. 

Similarly, applying de l'H\^opital to \eqref{eq:formo_integer1_limit1.} for 
$\beta\to0$ yields \eqref{eq:formo_integer1_limit2.}.
 \end{proof}

   \begin{theorem}
   Let $|\re(\alpha)|<1$ and $\re(\lambda)>0$. Then
  \begin{align}
    &      \label{formo2}
\int_1^\infty{\bf Z}_{\alpha,\lambda_1}(w){\bf
  Z}_{\alpha,\lambda_2}(w)(w^2-1)^\alpha\dd 2w\\\notag=&
   \frac{2^{\lambda_1+\lambda_2+1}
                  }{(\lambda_1^2-\lambda_2^2)\sin\pi\alpha}\Big(
                 \frac{1}{\Gamma(\frac12-\alpha+\lambda_1) 
          \Gamma(\frac12+\alpha+\lambda_2)}
          -(\lambda_1\leftrightarrow\lambda_2)\Big)
    , \end{align}
    \begin{align}
 &\int_1^\infty{\bf Z}_{0,\lambda_1}(w){\bf
  Z}_{0,\lambda_2}(w)\dd 2 w =
                                              \frac{2^{\lambda_1+\lambda_2+2} (\psi(\tfrac12 + \lambda_1) - \psi(\tfrac12 + \lambda_2))}{\pi (\lambda_1^2-\lambda_2^2) \Gamma( \tfrac12 + \lambda_1) \Gamma (\tfrac12 + \lambda_2)}, \\
& \int_1^\infty{\bf Z}_{\alpha,\lambda}(w)^2(w^2-1)^\alpha\dd 2 w 
= \frac{2^{2 \lambda } (\psi(\tfrac12 + \alpha + \lambda) 
- \psi(\tfrac12 -\alpha + \lambda))
}{\lambda \sin\pi\alpha \Gamma(\tfrac12 -\alpha + \lambda)
\Gamma(\tfrac12+\alpha+\lambda)}, 
 \\
&\int_1^\infty{\bf Z}_{0,\lambda}(w)^2 \dd 2 w = \frac{2^{2 \lambda+1}
\psi'(\tfrac12 + \lambda)}{\pi \lambda \Gamma(\tfrac12 + \lambda)^2}.
    \end{align}
\end{theorem}
\begin{proof}
\begin{align*}
    &(\lambda_1^2-\lambda_2^2) \int_1^\infty 
{\bf Z}_{\alpha,\lambda_1}(w){\bf 
      Z}_{\alpha,\lambda_2}(w)(w^2-1)^\alpha\dd w
 \\\notag=&
 \int_1^\infty
{\bf Z}_{\alpha,\lambda_1}(w) \partial_w (w^2-1)^{\alpha +1} 
\partial_w{\bf Z}_{\alpha,\lambda_2}(w)  \dd w
        \\ \notag &- \int_1^\infty
        \Big(\partial_w (w^2-1)^{\alpha+1}  
        \partial_w{\bf Z}_{\alpha,\lambda_1}(w) \Big)
        {\bf Z}_{\alpha,\lambda_2}(w)\dd w
\\
    =&\lim_{w\searrow1}\Big(
{\bf Z}_{\alpha,\lambda_1}(w) (w^2-1)^{\alpha +1} 
\partial_w{\bf Z}_{\alpha,\lambda_2}(w)
    - (w^2-1)^{\alpha+1} \big( \partial_w{\bf Z}_{\alpha,\lambda_1}(w)\big) 
    {\bf 
       Z}_{\alpha,\lambda_2}(w)\Big).
  \end{align*}
  Then we use the connection formula \eqref{formu1} to decompose the
  ${\bf Z}$ functions into the ${\bf S}$ functions. 
  The remaining formulas follow by the rule of de l'H\^opital.
   \end{proof}

\begin{remark}
The integrand ${\bf Z}_{\alpha,\lambda}(w)^2(w^2-1)^\alpha$ has 
asymptotics $\sim w^{-1-2\lambda}$ for $w\to\infty$ and therefore the case $\re(\lambda)\le0$ escapes our considerations.
\end{remark}

As before, the above integral formulas remain valid in the 
non-anomalous case $\alpha\in\cc\setminus\zz$ if the integral is 
replaced by the generalized integral. We compute the anomalous 
generalized integrals for $\alpha\in\zz$: 

\begin{theorem}
\label{thm:genZint}
Let $\alpha\in\zz$.  For $\re(\lambda_i)>0$, we have 
 \begin{align}
 \label{eq:gen_intZZ.}
 &  \frac{\pi}{2^{\lambda_1 + \lambda_2 +1}} \Big( \prod_{i=1}^2 \Gamma\big(\tfrac12+|\alpha|+\lambda_i\big) \Big)   \gen\int_1^\infty{\bf Z}_{\alpha,\lambda_1}(w){\bf
  Z}_{\alpha,\lambda_2}(w)(w^2-1)^\alpha\dd 2w
  \\ \notag
  = &  \Bigg(  \frac{(\frac12+\lambda_1)_{|\alpha|}(\frac12-\lambda_1)_{|\alpha|}}{\lambda_1^2-\lambda_2^2}
  \Big( -\ln(4) + \sum_{\pm} \psi\big(\tfrac12 \pm \alpha+\lambda_1\big) \Big)
  +(\lambda_1\leftrightarrow\lambda_2) 
    \Bigg)
  \\ \notag 
  &- (\tfrac12+\lambda_2)_{|\alpha|}(\tfrac12-\lambda_2)_{|\alpha|}
            \sum_{k=0}^{|\alpha|-1}
      \frac{(\frac12+\lambda_1)_k(\frac12-\lambda_1)_k}
{      (\frac12+\lambda_2)_{k+1}(\frac12-\lambda_2)_{k+1}}
 \\ \notag  &
\times \Big( -\psi(|\alpha|-k) -\psi(1+k) + \sum_{\pm}
   H_{|\alpha|-1-k}\big( \tfrac12-|\alpha| \pm \lambda_2 \big)  \Big).
\end{align}
 Moreover, if $\re(\lambda)>0$,
\begin{align}
\label{eq:genintZZ_limit.}
   & \frac{(-1)^{\alpha} \pi }{2^{2\lambda+1}}   \Big( \prod_{\pm} \Gamma\big(\tfrac12 \pm \alpha+\lambda\big) \Big)
     \gen\int_1^\infty{\bf Z}_{\alpha,\lambda}(w)^2(w^2-1)^\alpha\dd 2w 
 \\ \notag 
 =  & \frac{1}{2 \lambda} \sum_{\pm} \Big(
   \psi'(\tfrac12 \pm \alpha+\lambda) \mp H_{|\alpha|}(\tfrac12 \pm \lambda) \ln(4) \Big) \\
\notag&   + \sum_{k=0}^{|\alpha|-1} 
     \frac
   {\psi(\tfrac32+k+\lambda)
                 +\psi(-\tfrac12-k+\lambda)
   -\psi(|\alpha|-k)
   -\psi(1+k)
     }{\lambda^2-\big(\tfrac12+k\big)^2}.
 \end{align}
\end{theorem}
\begin{proof}
The proof works similar to the proof with the anomalous generalized 
${\bf S}$-integrals.
For  $\alpha\in\cc$ we set
\begin{align} 
f(w,\alpha) =\frac{\Gamma(\tfrac12-\alpha+\lambda_1) \Gamma(\tfrac12-\alpha+\lambda_2)}{2^{\lambda_1+\lambda_2-2\alpha-1}}
{\bf Z}_{\alpha,\lambda_1}(w) {\bf Z}_{\alpha,\lambda_2}(w)(w^2-1)^\alpha.\end{align}
  Then for $\alpha\not\in\zz$,
  \begin{align}\label{kokl}
\gen \int_1^\infty f(w,\alpha)\dd 2 w=
  \frac{2^{2\alpha+2}}{(\lambda_1^2-\lambda_2^2)\sin \pi\alpha}
    \Bigg(\frac{\Gamma(\frac12-\alpha+\lambda_2)}{\Gamma(\frac12+\alpha+\lambda_2)}
    -(\lambda_1\leftrightarrow\lambda_2) 
    \Bigg).\end{align}
  Let $m\in\zz$. The residue and the finite part of
\eqref{kokl} at $\alpha=m$ are
\begin{align}
  \res(m)&=
  \frac{(-1)^m2^{2m+2}}{\pi(\lambda_1^2-\lambda_2^2)}
    \Bigg(\frac{\Gamma(\frac12-m+\lambda_2)}{\Gamma(\frac12+m+\lambda_2)}
    -(\lambda_1\leftrightarrow\lambda_2) 
        \Bigg),
\end{align}
\begin{align}
  &\fp    \int_1^\infty f(w,\alpha)\dd 2 w=
                                          \frac{(-1)^m2^{2m+2}}{\pi(\lambda_1^2-\lambda_2^2)}\\
 \times                            \Bigg(\frac{\Gamma(\frac12-m+\lambda_2)}{\Gamma(\frac12+m+\lambda_2)} & \big(\ln(4) -\psi(\tfrac12-m+\lambda_2) -\psi(\tfrac12+m+\lambda_2)\big)
    -(\lambda_1\leftrightarrow\lambda_2) 
        \Bigg).\notag
  \end{align}
  For $\re(\alpha)\leq0$, the singular part of $f(\alpha,w)$ is
  contained in
  \begin{align}\notag
&f^\sing(\alpha,w)=\frac{\pi}{\sin^2\pi\alpha}{\bf S}_{\alpha,\lambda_1}(w) {\bf S}_{\alpha,\lambda_2}(w)(w^2-1)^\alpha\\
=&  \frac{ 2^{\alpha} \pi(w-1)^\alpha }{\sin^2\pi\alpha}
 {\bf 
  F}\Big(\frac12+\lambda_1,\frac12-\lambda_1,\alpha+1,\frac{1-w}{2}\Big)\notag
  \\ \notag &\times
 {\bf 
  F}\Big(\frac12+\alpha+\lambda_2,\frac12+\alpha-\lambda_2,\alpha+1,\frac{1-w}{2}\Big)  
\\\notag  =  &\frac{ 2^{\alpha} 
               (w-1)^{\alpha} }{\pi}
           \Bigg(\sum_{k=0}^\infty\frac{(\frac12+\lambda_1)_k 
         (\frac12-\lambda_1)_k\Gamma(-\alpha-k)}{k!}\Big(\frac{w-1}{2}\Big)^k\Bigg)\\
     &\times        \Bigg(\sum_{j=0}^\infty\frac{(\frac12+\alpha+\lambda_2)_j 
           (\frac12+\alpha-\lambda_2)_j\Gamma(-\alpha-j)}{j!}\Big(\frac{w-1}{2}\Big)^j\Bigg).
\label{summa1}  \end{align}
Let $m$ be a positive integer. The coefficient of \eqref{summa1} at
$\big(2(w-1)\big)^{\alpha+m-1}$ is
\begin{align}
& f_{m-1}(\alpha) 
\\ \notag = &\frac{2^{-2m+2}}{\pi} 
\sum_{k=0}^{m-1}\frac{(\frac12+\lambda_1)_k (\frac12-\lambda_1)_k
  (\frac12+\alpha+\lambda_2)_{m-1-k}
  (\frac12+\alpha-\lambda_2)_{m-1-k}}{k!(m-1-k)!}
\\\notag& \times\Gamma(-\alpha-k)\Gamma(-\alpha-m+1+k)  .
\end{align}
Now
\begin{align*}
&\frac{\dd}{\dd\alpha}f_{m-1}(\alpha)\Big|_{\alpha=-m}=\frac{2^{-2m+2}}{\pi}
\\\notag&\times  \sum_{k=0}^{m-1}(\tfrac12+\lambda_1)_k (\tfrac12-\lambda_1)_k
  (\tfrac12-m+\lambda_2)_{m-1-k}
                (\tfrac12-m-\lambda_2)_{m-1-k}\\ \notag
  &\times\Big(H_{m-1-k}(\tfrac12-m+\lambda_2)+H_{m-1-k}(\tfrac12-m-\lambda_2)
    -\psi(m-k)-\psi(1+k)\Big)
\end{align*}
Then we use
\begin{align}
 \big( \tfrac12 -m+\lambda_2 \big)_{m-1-k} 
 \big( \tfrac12 -m-\lambda_2 \big)_{m-1-k}=\frac{(\frac12+\lambda_2)_m (\frac12-\lambda_2)_m}{(\frac12+\lambda_2)_{k+1}(\frac12-\lambda_2)_{k+1}}.
  \end{align}

Thus we find the integral at negative integers:
\begin{align*}
  &\gen\int_1^\infty{\bf Z}_{-m,\lambda_1}(w)
 {\bf Z}_{-m,\lambda_2}(w) (w^2-1)^{-m}\dd 2w \\\notag
    = &\frac{2^{\lambda_1+\lambda_2-2m+1}}{\Gamma(\frac12+m+\lambda_1) \Gamma(\frac12+m+\lambda_2)}
        \gen\int_{-1}^1f(-m,w)\dd 2w\\\notag
  =& \frac{2^{\lambda_1+\lambda_2-2m+1}}{\Gamma(\frac12+m+\lambda_1) \Gamma(\frac12+m+\lambda_2)}
     \Bigg( \fp\int_{-1}^1f(-m,w)\dd 2w - \frac{\dd}{\dd\alpha}f_{m-1}(\alpha)\Big|_{\alpha=-m}\Bigg).\end{align*}
This proves
\eqref{eq:gen_intZZ.} for negative integers. But \eqref{eq:gen_intZZ.}
is invariant with respect to the flip of the sign of $\alpha$, because
\begin{align}
{\bf Z}_{-m,\lambda_1}(w) 
{\bf Z}_{-m,\lambda_2}(w) (w^2-1)^{-m}=
{\bf Z}_{m,\lambda_1}(w) 
{\bf Z}_{m,\lambda_2}(w) (w^2-1)^{m}
.\end{align}
Thus \eqref{eq:gen_intZZ.} has been proven.

To prove
\eqref{eq:genintZZ_limit.},
we set $\lambda=\lambda_1=\lambda_2$ in
\eqref{eq:gen_intZZ.}, using the
de l'H\^opital rule where needed:
\begin{align}
\label{eq:genintZZ_limit..}
   &\gen\int_1^\infty{\bf Z}_{\alpha,\lambda}(w){\bf
  Z}_{\alpha,\lambda}(w)(w^2-1)^\alpha\dd 2w 
   = 
  \frac{(-1)^{\alpha} 2^{2\lambda+1}}{\pi\Gamma\big(\tfrac12-\alpha+\lambda\big)
        \Gamma\big(\tfrac12+\alpha+\lambda\big)}
 \\ \notag 
  &\times \Bigg(
  \frac{\psi'\big(\tfrac12-\alpha+\lambda\big)
  +\psi'\big(\tfrac12+\alpha+\lambda\big)
    }{2\lambda}\\\notag
 &+ \frac{\big(\psi(\frac12-\alpha+\lambda)+\psi(\frac12+\alpha+
  \lambda)-\ln(4)\big)(-H_{|\alpha|}(\frac12-\lambda)+H_{|\alpha|}(\frac12+\lambda)\big)}
  {2\lambda} 
     \\&+
  \sum_{k=0}^{|\alpha|-1} 
     \frac{
  H_{|\alpha|-1-k}(\tfrac12-|\alpha|-\lambda)
                 +H_{|\alpha|-1-k}\big(\tfrac12-|\alpha|+\lambda\big)
   -\psi(|\alpha|-k)
   -\psi(1+k)
   }{\lambda^2-\big(\tfrac12+k\big)^2} 
    \Bigg).\notag
 \end{align}
Then we use
\begin{align}
\label{eq:sum_inverse_squares_lambda}
 \sum_{k=0}^{|\alpha|-1}\frac{1}{\lambda^2 -\big(\tfrac12+k\big)^2} 
  =&\frac{H_{|\alpha|}(\frac12+\lambda)-H_{|\alpha|}(\frac12-\lambda)}
{2\lambda},
\end{align}

\begin{align}
  \psi(\tfrac12-|\alpha|+\lambda)+H_{|\alpha|-1-k}(\tfrac12-|\alpha|+\lambda)
  &=\psi(-\tfrac12-k+\lambda),
\end{align}
\begin{align}
 &\psi(\tfrac12+|\alpha|+\lambda)+H_{|\alpha|-1-k}(\tfrac12-|\alpha|-\lambda)
 \\ \notag =&   \psi(\tfrac12+|\alpha|+\lambda)-H_{|\alpha|-1-k}(\tfrac32+k+\lambda) = \psi(\tfrac32+k+\lambda).
  \end{align}
  
 \end{proof}
  
\subsection{Asymptotics of bilinear integrals}

The bilinear (generalized) integrals for Gegenbauer functions converge to the
corresponding integrals for Macdonald functions consistently with
\eqref{asym3} and \eqref{asym1}. This is relatively straightforward in
the non-anomalous case. It is also true in the anomalous case, because we have 
 chosen  the right variables. Otherwise there would be a discrepancy.

\begin{theorem}
 \label{thm:asymp_genintSZ}
For $\beta \to \infty$ with fixed $\alpha$ satisfying $\re(\alpha)>-1$ we have
 \begin{align}
 \label{eq:Sint_asymp} 
&\frac{\pi^2 {\e}^{-2\pi\beta} \beta^{2\alpha}}{2^{2\alpha}}
  \;\gen\int_{-1}^1{\bf S}_{\alpha,\ii\beta}(w)^2(1-w^2)^\alpha\dd 2 w\\
  = &  \Big(1+\mathcal{O}\big(\tfrac{1}{\beta}\big)\Big)
    \;\gen\int_0^\infty  K_{\alpha}(\beta r)^2 2r \dd r,
  \notag
\end{align} 
For $\lambda \to \infty$ and all $\alpha \in \cc$ we have
\begin{align}
  \label{eq:Zint_asymp}
&\frac{\pi \Gamma\big(\tfrac12+\alpha+\lambda\big)^2
  }{2^{2\lambda+1}\lambda^{2\alpha}}
  \;\gen\int_1^\infty{\bf Z}_{\alpha,\lambda}(w)^2(w^2-1)^\alpha\dd 2 w \\
 = & \Big(1+\mathcal{O}\big(\tfrac{1}{\lambda}\big)\Big)
    \;\gen\int_0^\infty  K_{\alpha}(\lambda r)^2 2r \dd r  .
\notag  \end{align}
\end{theorem}
\begin{proof}
Let us first consider the non-anomalous case, that is
 $\alpha\in\cc\setminus\zz$.
For the $\mathbf{S}$-functions, we have
\begin{align*}
  &\LHS\eqref{eq:Sint_asymp}  \\
= & \frac{2\pi^2 \beta^{2\alpha}}{ {\e}^{2\pi\beta}}
\frac{  \ii \cosh(\pi\beta) 
    \big( \psi\big(\frac12 +\alpha + \ii \beta \big) 
                - \psi\big(\frac12 +\alpha - \ii \beta \big) 
           + \psi\big(\frac12 - \ii \beta \big) 
                - \psi\big(\frac12 + \ii \beta \big) \big)}
    {\beta\sin\pi\alpha \,\Gamma(\frac12+\alpha-\ii\beta)
    \Gamma(\frac12+\alpha+\ii\beta)}
\\
\,=\,  & 
\frac{   \pi \alpha}
    {\beta^2\sin\pi\alpha }
         \Big(1+\mathcal{O}\big(\tfrac{1}{\beta}\big) \Big)\,=
         \,  \RHS\eqref{eq:Sint_asymp} ,
\end{align*}
where we  used
\begin{align*}2 {\e}^{-\pi\beta} \cosh(\pi\beta) &
=  1 +\mathcal{O}\big(\tfrac{1}{\beta}\big),\\
\Gamma(\tfrac12+\alpha-\ii\beta)
\Gamma(\tfrac12+\alpha+\ii\beta)&=2\pi\beta^{2\alpha}\e^{-\beta\pi}(1+O(\tfrac1\beta)),\\
\psi(\tfrac12+\alpha+\ii\beta)- \psi(\tfrac12+\alpha-\ii\beta)
+\psi(\tfrac12-\ii\beta)-
\psi(\tfrac12+\ii\beta)&=-\frac{2 \ii\alpha}{\beta}+O(\tfrac1{\beta^2}).
\end{align*}
For the $\mathbf{Z}$-functions, we have
 \begin{align*}  
   &\LHS\eqref{eq:Zint_asymp}
  =  
  \frac{\pi \Gamma\big(\tfrac12+\alpha+\lambda\big)
  \big( \psi\big(\tfrac12+\alpha+\lambda\big)
       - \psi\big(\tfrac12-\alpha+\lambda\big) \big) 
  }{2 \lambda^{2\alpha+1} 
  \sin\pi\alpha\,\Gamma\big(\tfrac12-\alpha+\lambda\big)}
\\  = &\frac{\pi \alpha  }{\lambda^2 \sin\pi\alpha}
\Big(1+\mathcal{O}\big(\tfrac{1}{\lambda}\big)\Big)\,=\, \RHS\eqref{eq:Zint_asymp}
 \end{align*}
 where we used
 \begin{align*}\frac{\Gamma(\frac12+\alpha+\lambda)}{\Gamma(\frac12-\alpha+\lambda)}&=\lambda^{2\alpha}(1+O(\tfrac1\lambda)),\\[1.5ex]
\psi(\tfrac12+\alpha+\lambda)-\psi(\tfrac12-\alpha+\lambda)&=\frac{2\alpha}{\lambda}+O(\tfrac1{\lambda^2}) .\end{align*}

In the anomalous case, that is $\alpha\in \nn$, we have
 \begin{align}\notag
   &\LHS\eqref{eq:Sint_asymp}
 \\\notag 
= & \frac{(-1)^\alpha 4 \cosh(\pi\beta)\e^{-2\pi\beta}\beta^{2\alpha}}{\pi
\Gamma(\frac12+\alpha+\ii\beta)\Gamma(\frac12+\alpha-\ii\beta)}
      \Bigg( 
       \frac{\ii }{2 \beta } \Big(
       \psi'\big(\tfrac{1}{2}+\alpha +\ii \beta\big) 
    - \psi'\big(\tfrac{1}{2}+\alpha -\ii \beta\big)   
    \Big) 
    % \\
    %\notag
   %&-\frac{\ii}%{2\beta}\Big(H_{|\alpha|}\big(\tfrac12+\ii\beta)-H_{|\alpha|}\big(\tfrac12-\ii\beta)\Big)
   %\ln(4)
    \\   &\notag+
  \sum_{k=0}^{|\alpha |-1} 
 \frac{\psi \big(-\tfrac12-k+\ii \beta \big) 
    +  \psi\big(-\tfrac12-k-\ii \beta \big) 
    -\psi(|\alpha |-k)-\psi(1+k) - \ln 4 
 }{\big(\tfrac12 +k\big)^2+\beta^2
  }
           \Bigg)\\\notag
   =&(-1)^\alpha\Big(\frac{1}{\beta^2}-\frac{|\alpha|}{\beta^2}\ln(4)+\sum_{k=0}^{|\alpha|-1}\frac{\big(2\ln(\beta)-\psi(|\alpha|-k)-\psi(1+k)\big)}{\beta^2}\Big)\Big(1+O\big(\tfrac{1}{\beta}\big)\Big)\\\notag
   =&\RHS\eqref{eq:Sint_asymp},
\end{align}
where in the last equality we used \eqref{induction}. Similarly, for $\mathbf{Z}$ functions
\begin{align}\notag
  &\LHS\eqref{eq:Zint_asymp}
  \\\notag
  =&\frac{(-1)^{\alpha}\Gamma(\frac12+\alpha+\lambda)}{\Gamma\big(\tfrac12-\alpha+\lambda\big)
\lambda^{2\alpha}  } \Bigg(
  \frac{\psi'\big(\tfrac12-\alpha+\lambda\big)
  +\psi'\big(\tfrac12+\alpha+\lambda\big)
    }{2\lambda}
 \\ \notag 
&+\frac{H_{|\alpha|}(\tfrac12-\lambda)-
    H_{|\alpha|}(\tfrac12+\lambda)}{2\lambda}\ln(4)\\
\notag&   + \sum_{k=0}^{|\alpha|-1} 
\frac{ \psi(\tfrac32+k+\lambda)
                 +\psi\big(-\tfrac12-k+\lambda\big)
   -\psi(|\alpha|-k)
   -\psi(1+k)}{\lambda^2-\big(\tfrac12+k\big)^2} 
        \Bigg) \\\notag
   =&(-1)^\alpha\Big(\frac{1}{\lambda^2}-\frac{|\alpha|}{\lambda^2}\ln(4)+\sum_{k=0}^{|\alpha|-1}\frac{\big(2\ln(\lambda)-\psi(|\alpha|-k)-\psi(1+k)\big)}{\lambda^2}\Big)\Big(1+O\big(\tfrac{1}{\lambda}\big)\Big)\\\notag
  =&\RHS\eqref{eq:Zint_asymp}.
 \end{align}
 \end{proof}

\appendix

\section{Around the Gamma function}
\label{app:Gamma}
The {\em Gamma function} and its logarithmic derivative, the {\em digamma
function}, play a central role in the theory of special
functions. In this section we collect properties of these functions,
which we will often use in our paper. We will also describe the
properties of the
{\em Pochhammer 
symbol} and {\em harmonic numbers}, which are also closely related to the
Gamma and digamma functions.

\subsection{Definitions}
Recall the definition of  the digamma function,
\begin{align}
 \psi(z)&:= \partial_z \ln\Gamma(z) 
 = \frac{\partial_z\Gamma(z)}{\Gamma(z)}.\label{digamma}
\end{align}
It is also useful to introduce the Pochhammer symbol,
\begin{align}
 (z)_k&:=\frac{\Gamma(z+k)}{\Gamma(z)}
\label{pochhammer1}
=\begin{cases}(z)(z+1)\cdots(z+k-1),& k\geq0,\\
\frac{1}{(z+k)(z+k+1)\cdots(z-1)},&k\leq0;\end{cases}
\end{align}
\begin{align}
 \textup{the shifted $k$th harmonic number}&&    
    H_k(z)& := \frac{1}{z} + \dots + \frac{1}{z+k-1},\label{haka}\\
 \textup{and the $k$th harmonic number}&&                          H_k&:=
   H_k(1)=\frac{1}{1} + \dots + \frac{1}{k}.\label{haka0}
\end{align}

\subsection{Basic properties}

\begin{align}
  \Gamma(z+1)&=z\Gamma(z),
 \\
 \frac{1}{\Gamma(1-z)\Gamma(z)}
             &= \frac{\sin\pi z}{\pi} ,
               \\
 \frac{1}{\Gamma\big(\tfrac12+z\big)\Gamma\big(\tfrac12-z\big)}
             &= \frac{\cos\pi z}{\pi},\\
  \frac{\Gamma(2\lambda+1)}{\Gamma(\frac12+\lambda)\Gamma(1+\lambda)}&=\frac{2^{2\lambda}}{\sqrt\pi}.\label{dupli}
\end{align}

$\frac{1}{\Gamma(z)}$ is an entire function with zeros at 
$0,-1,-2, \dots$ and derivative
 \begin{align}
   \partial_z\frac{1}{\Gamma(z)}&=-\frac{\psi(z)}{\Gamma(z)}.
 \end{align}
 The digamma function satisfies the following functional relations 
 \cite[Chapter 8.365]{GR}:
\begin{align}
 \psi(1+z)&=\psi(z)+\frac1z,\\
  \psi(z)-\psi(1-z)&=-\pi\cot(\pi z),\\
 \psi\big(\tfrac12+z\big)- \psi\big(\tfrac12-z\big)&=\pi\tan(\pi z),  \\
 2\psi(2z) &= 2\ln2+\psi(z)+\psi\big(z+\tfrac12\big), \\
 \psi(z+k)&=\psi(z)+H_k(z),\\
 \psi(1+k)&=-\gamma_{\textup{E}}+H_k.
 \label{psiha}
\end{align}
The following relations among the harmonic numbers and 
the Pochhammer symbol are immediate from there definitions:
\begin{align}
 H_{k+n}(z)&=H_n(z)+H_k(z+n),\\
  H_k(z)&=-H_k(1-z-k),\\
  (z)_k(z+k)_n&=(z)_{k+n},\\
 (z)_k &= (-1)^k(1-k-z)_k,\label{pochhammer}\\
  (z)_{-k}&=\frac{1}{(z-k)}_k=\frac{(-1)^k}{(1-z)_k},\\
  (\tfrac12+z)_k(\tfrac12-z)_k&=(-1)^k(\tfrac12+z-k)_{2k},\\
 \partial_z(z)_k&=H_k(z)(z)_k.
\end{align} 
The Pochammer symbol is also useful to expand powers around a different 
center:
\begin{align} 
 (1-z)^{-a}&=\sum\limits_{k=0}^\infty\frac{(a)_k}{k!}z^k,\ \ \ |z|<1.\label{double1}
\end{align}
The following identities follow by induction:
\begin{align}\label{induction}
  \sum_{k=0}^{m-1}\psi(1+k)&=m\big(\psi(1+m)-1\big),\\
  \sum_{k=0}^{m-1}H_k&=m\big(H_m-1\big).
                       \end{align}

\subsection{Special values}
We have \cite[Chapter 8.366]{GR}
\begin{align}
 \psi(1)=-\gamma_{\textup{E}},\qquad
 \psi\big(\tfrac12\big)=-\gamma_{\textup{E}}-2\ln2, \qquad
 \psi'\big(\tfrac12\big) = \frac{\pi^2}{2},
\end{align}
with the Euler-Mascheroni constant $\gamma_{\textup{E}}$. Moreover, 
we have
\begin{align}
(1/2)_k = \frac{\Gamma(k+\frac12)}{\sqrt\pi}=\frac{(2k)!}{2^{2k}k!}=
\frac{(2k-1)!}{2^{2k-1}(k-1)!}  \label{double} 
\end{align}
and 
\begin{align}
  \partial_z\frac{1}{\Gamma(z)}\Big|_{z=-n}&=(-1)^nn!,\quad n=0,1,2,\dots
\end{align}

\subsection{Asymptotics for large arguments}

Let $\epsilon>0$. Then for $|\arg(z)|<\pi-\epsilon$, as $|z|\to\infty$,
we have
\begin{align}\label{stirling}
  \Gamma(z)&=z^{z-\frac12}\e^{-z}\sqrt{2\pi}\Big(1  +\mathcal{O}\big(\tfrac{1}{z}\big)\Big),\\
 \ln \Gamma(z) &= z \ln z - z +\frac{1}{2} \ln \frac{2\pi}{z} 
                 +  \mathcal{O}\big(\tfrac{1}{z}\big),\label{stirling1}\\
  \psi(z)&=\ln(z)-\frac{1}{2z}+ \mathcal{O}\big(\tfrac{1}{z^2}\big),\label{stirling2}\\
    \psi'(z)&=\frac1z+\frac{1}{2z^2}+ \mathcal{O}\big(\tfrac{1}{z^3}\big).\label{stirling3}
\end{align}
\eqref{stirling} is the famous Stirling formula.
\eqref{stirling1} is obviously equivalent to
\eqref{stirling}. Formally, 
\eqref{stirling2} and \eqref{stirling3} follow from
\eqref{stirling1}  by differentiation. However, rigorously, we are not
allowed to differentiate asymptotic series, and therefore
\eqref{stirling2} and \eqref{stirling3} require an independent
proof. We can use for this the so-called 2nd Binet's identity
\cite{WW} and its derivatives:
\begin{align}
\ln\Gamma(z)&=\Big(z-{\frac12}\Big)\ln z-z+{\frac12}\ln 2\pi
+2\int_0^\infty
\frac{\arctan\frac{t}{z}}{\e^{2\pi t}-1}\dd t
,\label{Binet1}\\
\psi(z) &
=\ln z-\frac{1}{2z}-2\int_0^\infty\frac{t\dd t}{(z^2+t^2)(\e^{2\pi t}-1)}
,\label{Binet2}\\
\psi'(z) & =\frac{1}{z}+
\frac{1}{2z^2}
+4\int_0^\infty\frac{zt\dd t}{(z^2+t^2)^2(\e^{2\pi t}-1)}
.
\label{Binet3}\end{align}
Let us show e.g. that \eqref{Binet1} implies \eqref{stirling1}. Let
$0<2\theta<\pi$.
We have $|\arctan u|\leq c|u|$ for $|\arg u|<\theta$. We set
$t=\e^{\ii\theta}s$. Then we use
\begin{align}
\Bigg|\frac{\arctan\frac{t}{z}}{\e^{2\pi t}-1}\Bigg|\leq\frac{c}{|z|}
\frac{1}{\e^{2\pi s\cos\theta}}
\end{align}

One can present the asymptotics \eqref{stirling1}, \eqref{stirling2}
and \eqref{stirling3}
in a different equivalent form. Let us fix $\alpha$ and
let $|\arg\lambda|<\pi-\epsilon$. Then as $|\lambda|\to\infty$, we have
\begin{align}
 \ln \Gamma(\tfrac12+\alpha+\lambda) &= (\alpha+\lambda) \ln \lambda - \lambda +\frac{1}{2} \ln (2\pi )
                 +  \mathcal{O}\big(\tfrac{1}{\lambda}\big),\label{stirling1.}\\
  \psi(\tfrac12+\alpha+\lambda)&=\ln(\lambda)+\frac{\alpha}{\lambda}+ \mathcal{O}\big(\tfrac{1}{\lambda^2}\big),\label{stirling2.}\\
    \psi'(\tfrac12+\alpha+\lambda)&=\frac1\lambda-\frac{\alpha}{\lambda^2}+ \mathcal{O}\big(\tfrac{1}{\lambda^3}\big).\label{stirling3.}
\end{align}
  Here is a consequence of the Stirling formula \cite[Lemma 2.1]{CDT}:
\begin{align}
\label{eq:asymp_gamma_ratio_imaginary}
\frac{\Gamma(a+\lambda)}{\Gamma(b+\lambda)} 
&=  \lambda^{a-b} 
\Big(1  +\mathcal{O}\big(\tfrac{1}{\lambda}\big)\Big). 
\end{align}
Here is another consequence, where $\beta>0$ and $\beta\to\infty$:
\begin{align}
\Gamma(\alpha+\ii\beta) \Gamma(\alpha-\ii\beta)=2\pi\beta^{2\alpha-1}\e^{-\pi\beta}
\Big(1  +\mathcal{O}\big(\tfrac{1}{\beta}\big)\Big). 
\end{align}

\section{Other conventions}
\label{app:assoc_legendre}

This appendix is included to facilitate the comparison of our formulas 
with those contained in a part of the literature.

\subsection{Associated Legendre equation.}
In the literature (see for example \cite{CDT,NIST,Olver}), instead of
the Gegenbauer equation (\ref{gege0}), 
one often considers the {\em associated Legendre equation} given by the operator
\begin{align}\label{geg1}
&
(z^2-1)^{\frac{\alpha}{2}}\Big((1-z^2)\partial_z^2-2(1+\alpha )z\partial_z 
+\lambda ^2-\Big(\alpha +\frac{1}{2}\Big)^2\Big)
(z^2-1)^{-\frac{\alpha}{2}}
\\ \notag
=&
\partial_z(1-z^2)\partial_z-\frac{\alpha^2}{1-z^2}+\lambda^2-\frac{1}{4}.
\end{align}
In the standard literature the parameter
$\lambda$  is shifted by $\frac12$. Thus, with $\mu=\alpha$ and $\nu=\lambda-\frac12$, (\ref{geg1}) becomes
\begin{align}\label{asso}
(1-z^2)\partial_z^2-2z\partial_z-\frac{\mu^2}{1-z^2}+\nu(\nu+1).
\end{align}
Certain functions annihilated  by \eqref{asso}
are called {\em associated Legendre functions}.
In this subsection we show how  to convert the functions ${\bf S}$ and
${\bf Z}$  into associated Legendre functions. The material of this subsection will not be used elsewhere in this paper.

{\em The associated Legendre function of the 1st kind} are defined as
\begin{align}
{\bf P}_\nu ^\mu (z)&=\left(\frac{z+1}{z-1}\right)^{\frac{\mu }{2}}
{\bf F} \Big(-\nu ,\nu +1;1-\mu ;\frac{1-z}{2}\Big) \notag\\
&=\frac{2^\mu }{(z^2-1)^{\frac{\mu }{2}}}
{\bf F} \Big(1-\mu +\nu ,-\mu -\nu ;1-\mu ;\frac{1-z}{2}\Big)
\label{eq:ferrersP}\\
&=\frac{2^\mu }{(z^2-1)^{\frac{\mu }{2}}}
{\bf S}_{-\mu ,\nu +\frac12}(z). \notag\end{align}

The functions ${\bf P}_\nu ^\mu$ are well adapted to the halfline $[0,\infty[$.
On the interval $]-1,1[$ one prefers to use {\em Ferrer's
 associated Legendre function of the 1st kind}
\begin{align} \notag
{\rm P}_\nu ^\mu (z)
&=\left(\frac{z+1}{1-z}\right)^{\frac{\mu }{2}}
{\bf F} \Big(-\nu ,\nu +1;1-\mu ;\frac{1-z}{2}\Big)\\ \notag
&=\frac{2^\mu }{(1-z^2)^{\frac{\mu }{2}}}
{\bf F} \Big(1-\mu +\nu ,-\mu -\nu ;1-\mu ;\frac{1-z}{2}\Big)
\\ \label{ZPrelation}
&=\frac{2^\mu }{(1-z^2)^{\frac{\mu }{2}}}
{\bf S}_{-\mu ,\nu +\frac12}(z).\end{align}

 {\em The associated Legendre function of the 2nd kind} are given by
\begin{align}
\label{eq:ferrersQ}
{\bf Q}_\nu ^\mu (z)
&=\e^{\mu\ii\pi}\sqrt{\pi}\Gamma(\mu+\nu+1)
\frac{(z^2-1)^{\frac{\mu }{2}}}{2^{\nu +1}z^{\nu +\mu +1}} {\bf F} \Big(\frac{\nu +\mu +2}{2},\frac{\nu +\mu +1}{2};\nu +\frac{3}{2};z^{-2}\Big)
\\ \notag
&= \frac{\e^{\mu\ii\pi}\sqrt{\pi}\Gamma(\mu+\nu+1)
(z^2-1)^{\frac{\mu }{2}}}{2^{\nu +1}(1+z)^{\nu +\mu +1}\Gamma(\nu +\frac{3}{2})}  F \Big(\nu +1,\nu +\mu +1;2\nu +2;\frac{2}{z+1}\Big)
\\ \notag
&= \e^{\mu\ii\pi}\sqrt{\pi}\Gamma(\mu+\nu+1)\frac{(z^2-1)^{\frac\mu2}}{2^{\nu+1}}
  {\bf Z}_{\mu,\nu+\frac12}(z).
\end{align}

\subsection{Other conventions for Gegenbauer functions.}
One sometimes finds other conventions for the Gegenbauer functions in 
the literature \cite{Durand76, Durand19a,NIST}: 
\begin{align}
 C^\alpha_\lambda(z) &= \frac{\Gamma(\lambda+2\alpha)}{\Gamma(1+\lambda)\Gamma(2\alpha)}
 F\big(-\lambda,\lambda+2\alpha;\alpha+\tfrac12;\tfrac{1-z}{2}\big),
 \\
 D^\alpha_\lambda(z) &= \frac{\e^{\ii\pi\alpha} \Gamma(\lambda+2\alpha) (2z)^{-\lambda-2\alpha}
 }{\Gamma(\alpha)\Gamma(1+\alpha+\lambda)}
 F\big(\tfrac12\lambda+\alpha,\tfrac12+\tfrac12\lambda+\alpha;1+\alpha+\lambda;z^{-2}\big),
\end{align}
where the second identity is valid for $|z|>1$.
Comparing these definitions with \eqref{solu1}, \eqref{solu3}, 
\eqref{eq:ferrersP} and \eqref{eq:ferrersQ}, we find 
\begin{align}
 \frac{\Gamma(\alpha)\Gamma(1+\lambda)}{\sqrt{\pi}\Gamma(\lambda+2\alpha)} C^{\alpha}_{\lambda}(z) 
 &= 2^{\tfrac12-\alpha} 
 (z^2-1)^{\tfrac{1}{4}-\tfrac{\alpha}{2}} {\bf P}^{\tfrac12-\alpha}_{-\tfrac12+\alpha+\lambda}(z)
 \\ \notag &=  2^{1-2\alpha} 
 {\bf S}_{\alpha-\tfrac12,\alpha+\lambda}(z),
 \\ 
  \frac{\Gamma(\alpha)\Gamma(1+\alpha+\lambda)}{\Gamma(\lambda+2\alpha)} D^{\alpha}_{\lambda}(z) 
 &=
 \frac{2^{\tfrac12-\alpha} \e^{\ii\tfrac\pi2} (z^2-1)^{\tfrac{1}{4}-\tfrac{\alpha}{2}}}{\sqrt{\pi}\Gamma\big(\tfrac12+2\alpha+\lambda\big)}  {\bf Q}^{\alpha-\tfrac12}_{-\tfrac12+\alpha+\lambda}(z)
 \\ \notag
 &= \e^{\ii\pi\alpha} 2^{-\lambda-2\alpha} Z_{\alpha-\tfrac12,\alpha+\lambda}(z).
\end{align}

%\marginpar{{\color{blue} 1. Added Lesch to the refs. 
%2. Updated title of pt-pot manuscript. Should we aim for 
%pt-pot manuscript having arxiv-identifier when we 
%submit revised Gegenbauer? 3. Added Bollini Giambiagi for dim 
%reg in QFT}}
\bibliographystyle{ABMR}

\end{document}